\documentclass[11pt]{article}

\usepackage{color}
\usepackage{latexsym,dsfont}
\usepackage{amssymb}
\usepackage{amsmath}
\usepackage{bbm}
\usepackage{amsthm}
\usepackage{graphicx}
\usepackage{epsfig}
\usepackage{float}
\usepackage{stmaryrd}
\usepackage{fancyhdr}
\usepackage{enumerate}
\usepackage{hyperref}

\newtheorem{Theorem}{Theorem}[part]

\newtheorem{Proposition}{Proposition}[part]

\newtheorem{Remark}{Remark}[part]

\def \Sum{ \sum}

\def\1{\mathbbm{1}}

\def \N{\mathbb{N}}
\def \R{\mathbb{R}}

\def \F{\mathbb{F}}

\def \Gc{{\cal G}}

\def\beqs{\begin{eqnarray*}}
\def\enqs{\end{eqnarray*}}
\def\beq{\begin{eqnarray}}
\def\enq{\end{eqnarray}}

\newcommand{\nc}{\newcommand}
\nc{\esssup}{\mathop{\mathrm{ess\,sup}}}

\newlength{\llegende}
\definecolor{Gcolor}{rgb}{1,0,0.5}

\begin{document}
\title{Dam Management in the Era of Climate Change}
\author{
Cristina Di Girolami\footnote{Dipartimento di Matematica, Universit\`a Alma Mater Studiorum Bologna, Piazza di Porta San Donato, 5, 40126 Bologna, Italy, email: cristina.digirolami2@unibo.it}, 
M'hamed Ga\"igi\footnote{Universit\'e de Tunis El Manar,
Ecole Nationale d'Ing\'enieurs de Tunis, ENIT-LAMSIN, B.P. 37,
Tunis, 1002, Tunisia; email: mhamed.gaigi@enit.utm.tn}, Vathana Ly Vath\footnote{ENSIIE, Laboratoire de Math\'ematiques et Mod\'elisation d'\'Evry, Universit\'e Paris-Saclay, CNRS UMR 8071, UEVE I.B.G.B.I., 23 Bd. de France, 91037 \'Evry Cedex, email: vathana.lyvath@ensiie.fr}, 
Simone Scotti\footnote{Dipartimento di Economia e Management, Universit\`a di Pisa, via Ridolfi 10, 56125 Pisa; email:  simone.scotti@unipi.it}
}

\maketitle

{\footnotesize
{\bf Abstract} {Climate change has a dramatic impact, particularly by concentrating rainfall into a few short periods, interspersed with long dry spells. In this context, the role of dams is crucial.
We consider the optimal control of a dam, where the water level must neither exceed a designated safety threshold nor fall below a minimum level to ensure functionality and sustainability for the downstream river.
To model dry spells and intense rainfall events, commonly referred to as water bombs, we introduce a Hawkes process, a well-known example of a self-exciting process characterized by time-correlated intensity, which endogenously reproduces the clustering of events. The problem is formulated as an optimal switching problem with constraints. We establish existence results and propose numerical methods for approximating the solution.
Finally, we illustrate the main achievements of this approach through numerical examples focusing in particular on the sensitivity of the self-exciting parameter describing the importance of both water bombs and dry-spells.
For the parameter configurations considered in this paper, the optimal water level inside the dam decreases as the self-exciting parameter increases. This numerical finding suggests that, under the considered calibration, the management response is driven more strongly by overtopping risk than by drought risk.
In conclusion, dams will increasingly lose their role as water reserves and take on a greater role in flood protection.}
\medskip




\section{Introduction}

Climate change has dramatic impacts on people's lives and the global economy. These impacts are, however, quite peculiar: rather than a simple rise in global temperatures, they are characterised by an exponential increase in extreme events. In this paper, we focus mainly on freshwater, which is the keystone of life and the economy, as it is used for drinking, agriculture, industry, and other purposes. Freshwater is essentially generated by rainfall, and its management was the cornerstone of some of the most magnificent historical civilizations in the world, such as the Romans, with their construction of aqueducts, and the Chinese, Babylonians, Maya, and Khmer, with the construction of channels, reservoirs, and lakes (see, for instance, \cite[Chapters 26 and 35]{Minghi}).
More recently, the role of dams in storing freshwater and producing electricity has enabled the settlement of vast desert areas around the world, such as the Colorado Plateau with the construction of the Hoover Dam, and has allowed the harnessing of monsoonal rains in tropical regions. Local populations have adapted their lives to the specific timing and periodicity of rainfall. However, current climate disruption has placed populations in a position where they must find new solutions to manage water flows. In particular, a new and significant risk arises: rainfall, which was previously relatively homogeneous across years (ignoring natural seasonality), now exhibits extreme fluctuations, with certain exceptional events clearly identified, such as the El Nino/La Nina phenomena (see, for instance, \cite{Timmermann}).

In the new century, this relative stability has completely broken down, with the occurrence of long droughts, such as the 2011-2017 drought in California, followed by catastrophic rainfall events, such as the winter of 2017, which was the wettest on record. The resulting floodwaters caused severe damage to the Oroville Dam in Northern California.
The importance of modelling rainfall dynamics has grown in recent years, not only to capture droughts, even in temperate climate zones such as Western Europe, North America, and East Asia (see \cite{Hottovy, Miller2017}), but also to account for the dramatic increase in so-called ``water bombs", i.e. episodes of intense rainfall over a short period. Another important, yet often overlooked, aspect of dam inflow is its direct relationship with climate change. Specifically, the inflow to a mountain dam depends on both rainfall and snowmelt (see \cite{Miller2017}).
Historically, in Europe, rainfall was more prominent in spring, especially in April, while snowmelt peaked in early June. This seasonal pattern ensured a prolonged period of significant inflows, enabling reservoirs to replenish before summer, which is typically characterised by minimal or no rainfall. However, global warming has advanced the timing of snowmelt to April or May, compelling dam managers to release more water into downstream rivers during the spring. This shift increases the risks of flooding in spring and drought in summer. One consequence of this phenomenon is that Switzerland has decided to increase the height of certain dams by approximately 20 meters (see \cite{Manso}), incurring significant maintenance and risk costs.

Suitable mathematical models for dam management are of critical importance for technological, economic, and ecological reasons (see \cite{Abdel-Hameed, Brockwell, Bather, Cuvelier, Meyerhoff, Roche, Yeh1}). Moreover, the significant environmental impact of water reservoirs underscores the need for reliable models for managing and controlling dams and water flows (see \cite{Roche, Yeh2}).
Strict regulations govern water flow management through dams to ensure both safety and efficiency. These regulations are often expressed via the so-called ``Rule Curves" (see \cite{Prasanchum}). A substantial body of literature addresses dam management as an optimal control problem; see \cite{Faddy} for a comprehensive survey.

In this paper, we focus on dam management as a cornerstone of water management and electricity production in temperate regions. Due to freshwater scarcity, the water stored in artificial lakes must be carefully controlled to achieve multiple objectives. First, a minimal outflow must be ensured during drought periods.
A second, but equally important, role of dams is that hydroelectric plants represent a perfectly green and renewable energy source, unlike wind and solar, which are intermittent \cite{Alonzo}, or nuclear plants, which produce radioactive waste. Hydroelectric plants thus play a crucial role both in decarbonisation and in stabilising the electrical grid. Mitigation of carbon emissions can be achieved by replacing fossil-fuel power plants with renewable installations, which has a significant impact on the CO? market (see \cite{Krach-Macrina-24}) and contributes to the global growth of dam construction, with projects increasing in both number and size.
However, the non-intermittency of hydropower is increasingly challenged by extreme events such as droughts and heavy rainfall, which force operators to manage water flows that are increasingly irregular and often extreme.

In mathematical point of view, we study the dam control problem as an optimal mixed switching problem. The dam manager must optimise revenues from power production while maintaining the water level within safe limits. The main challenge lies in managing water inflows, which we model using a marked Hawkes process. This process describes inflow dynamics as a marked jump process, where the intensity itself experiences sudden changes. This framework captures extended periods of dry spells and clustered rain bombs endogenously, rather than assuming rainfall to be homogeneously distributed over time.

Hawkes processes \cite{Haw71} are well-suited for modeling such behaviors and are increasingly applied in various domains, including mathematical finance (see \cite{ACL15, BerBriScoSga21, EGG10, JMS21, MP22}), water management (see \cite{MoralesETC, YY22a, YY22b}), and, more recently, cyber risk (see \cite{BBH21, HLOS22, HL21}). The main advantage of Hawkes processes, compared to Cox processes \cite{Lando} or hidden Markov chains \cite{Mamon}, lies in their parsimony and the fact that the process is directly observable, which makes estimation easier.
A simple example of a Hawkes process is a one-dimensional pure jump process, where the intensity jumps simultaneously with the points of the process. Consequently, the intensity itself is a one-dimensional Markov process with observable large jumps, and only the jump law and mean-reversion speed need to be estimated (see, for instance, \cite{BerBriScoSga21, CMS}). In \cite{HwRains}, the authors motivate modeling rainfall with marked Hawkes processes with uncontrolled density, due to the nature of the problem. See Ogata \cite{Ogata} for estimation methods and Dassios and Zhao \cite{DZ} for simulation procedures.
Both simulations \cite{DZ} and theoretical results \cite{CMS, JMSS19} show that Hawkes intensity naturally exhibits a two-phase structure: short periods of high intensity, characterized by many concentrated jumps, interspersed with very long periods of low intensity and few jumps. This feature allows the Hawkes setup to capture both extreme rainfall events and prolonged periods of drought.

A particularly challenging constraint in electricity production arises from the requirement to maintain a fixed grid frequency. In cases where frequency differs across regions, such as in Brazil and Paraguay, the hydroelectric plant at Itaipu is split into two parts: the Paraguayan side operates at 50 Hz, while the Brazilian side operates at 60 Hz.
This issue, which dates back to the early 20th century, arises because the power grid and all connected devices are designed to operate at a specific, unified frequency. Maintaining this frequency imposes an implicit constraint on production units, often requiring an on-off strategy for electricity generation. Moreover, since the produced electricity depends solely on the turbine's rotational speed, within certain physical and engineering limits, it is clear that electricity production is effectively fixed when the turbine is running. Water consumption is then adjusted to maintain the rotational speed using a servomechanism, as described in \cite{ChDiGaGiSc}.

The optimal frequency is maintained by regulating water flows using servomechanisms. The control variables include the status of the turbines (open or closed), the switching times, and the operation of spillways, which allow water to pass through without generating electricity. Optimal switching and stopping problems have been extensively studied in mathematical finance, particularly in the context of investment policies (see, for instance, \cite{chegaily, ChLyRocSco, ChDiGaGiSc, LyVath-Pham, LyVath-Pham-Villeneuve}).

This article represents a step forward compared to \cite{ChDiGaGiSc}, essentially by introducing and emphasising a discontinuous SDE to model the water level in the reservoir.
In \cite{ChDiGaGiSc}, a similar setup is considered: a dam produces electricity, and the manager seeks to maximise revenue from power production while keeping the water level under control. However, in \cite{ChDiGaGiSc}, the water level was modeled solely as a diffusion process, i.e., a continuous process, and the problem was formulated as an optimal switching problem. In this manuscript, to study the effects of discontinuous contributions more deeply, we do not include a Brownian component; the SDE for the water level is driven exclusively by a marked Hawkes-type jump process.

The two modeling approaches can be justified by the existence of different types of dams worldwide. In particular, large-scale dams with very large catchment areas, such as the Three Gorges Dam in China or the Hoover Dam on the Colorado River, are naturally associated with water levels that evolve in a relatively smooth and continuous manner. By contrast, small dams located in alpine regions have very small catchment areas, where local rainfall events can induce abrupt and significant variations in water levels.
Consequently, although this paper and \cite{ChDiGaGiSc} share some aspects, they focus on two different types of dams with very different parameters and modeling setups. Moreover, the most relevant innovation of this paper is the clustering effect captured by Hawkes processes, a feature that cannot be easily incorporated into a Brownian framework.

Consequently, the stochastic intensity of the Hawkes process describing rainfall must be treated as a state variable. In contrast, we do not consider the electricity price as a state variable, as was done in \cite{ChDiGaGiSc}. Including the price as a third state variable could be handled exactly as in \cite{ChDiGaGiSc} using a numeraire technique, which reduces the state dimension from three (electricity price, reservoir level, and intensity) to two (reservoir level and intensity). For the sake of readability and to avoid unnecessary repetition of results already developed in \cite{ChDiGaGiSc}, in this paper we present the model directly in two dimensions. The change of numeraire affects only the monetary quantities appearing in the objective functional, namely revenues, switching costs and the value function itself. Physical quantities such as the reservoir level, rainfall intensity, spillway controls and environmental-flow constraints remain unchanged.
The dam manager must maximise revenue from power production while ensuring that water levels remain within prescribed limits: they cannot exceed a certain threshold for safety reasons, nor fall below a minimum threshold. Moreover, during long periods of drought, the dam must guarantee a minimal outflow to preserve biodiversity in rivers and support agriculture in the downstream basin.

Thus, under the same assumptions as in \cite{ChDiGaGiSc}, and motivated by the need to emphasise the effects of drought, it is reasonable to reduce the dimensionality of the problem to a two-dimensional state space by interpreting the Hawkes process intensity as a time clock.
Hawkes processes represent the main novelty of this paper, and we particularly highlight the impact of the self-exciting parameter on the optimal management policy. Specifically, we focus on the optimal water level in the artificial lake behind the dam, defined as the point at which it is optimal to operate the spillway system.

We emphasize that the self-exciting parameter increases both the magnitude of extreme events, such as heavy rainfall during a ``water bomb," and the duration of very low-intensity periods, such as dry spells (see, for instance, Jiao et al. \cite[Section 5]{JMS17}). Consequently, both the risk of dam overtopping and the risk of water shortages are amplified as the self-exciting parameter increases, reflecting the current challenges posed by climate change. However, the optimal dam management response is not straightforward: overtopping risk calls for lowering the reservoir level, while prolonged dry spells necessitate maintaining a higher water level to mitigate drought effects.

For the parameter configurations considered in our numerical experiments, the optimal policy reacts more strongly to overtopping risk than to drought risk:
the optimal water level in the artificial lake decreases, even though the risk of water scarcity is simultaneously amplified. In other words, the dual role of the dam, as both flood protection and water supply, is disrupted under the current context of climate change, where prolonged dry spells become increasingly likely.
%
%

We characterize the optimal policy using viscosity solutions of the associated system of Hamilton-Jacobi-Bellman partial integro-differential inequalities. The solution is then obtained using a numerical scheme adapted from controlled Markov chain problems, see \cite{KushDup}.
The paper is organized as follows: in Section \ref{Model}, we introduce the model, in Section \ref{The Optimal Control Problem} we formulate the optimal control problem with the associated HJB equations. Section \ref{optimal-switching} focuses on the optimal switching formulation. In Section \ref{numericalResu}, we present numerical examples and, finally, we conclude with Section \ref{concluding-remarks}.
 
 \section{The rainfalls and dam model}\label{Model}
 
 Let us consider a dam and denote the water level inside the dam $H_t$ as a function of time.
We consider a filtered probability space $(\Omega, \mathcal{F},  \mathbb{F}=\left(  \mathcal{F}_t\right)_{t\geq 0}, \mathbb{P})$ satisfying the usual hypotheses, equipped with a point process.

Let be given the following constants belonging to the characteristic of the power production dam: $h_0<0<h_-<h_+<h^{max}$. $h_0$ denotes the position of the turbine with respect to the  bottom of the reservoir and it is then generally negative, 
$h_-$ and $h_+$ are the safety thresholds, respectively dangerous minimal and maximal water level of the dam, $h^{max}$ denotes the maximal water level inside the dam.

Focusing on mountain dams, the water height $\{H_t\}_{t\geq 0}$ is affected by the external contribution provided, essentially, by rain and torrents which flow is extremely uneven. The inflow is often negligible during dry periods and huge when storms occur. 
The climate change has moreover increased the length of dry periods and storms occur by clusters as the old saying ``when it rains, it pours".
For this reason we replace the usual Brownian motion for the inflow, see for instance Chevalier et al. \cite{ChDiGaGiSc}, by a compounded process $\{Z_t\}_{t\geq 0}$ driven by a counting process $\{J_t\}_{t\geq 0}$ with a Hawkes intensity $\{\lambda_t\}_{t\geq 0}$, that is the process $J$ counts the jump times $(\Theta_i)_{i\in \N}$, i.e. $J_t=\sum_{n\geq 1}\1_{t\geq \Theta_n}$\footnote{Notice that the counting process $J$ is defined via its intensity $\lambda$ , which
in turn depends on the history of $J$. So an apparent logical loop seems to arise
about the existence of $\lambda$. We refer to Brachetta et al \cite[Section 2.1]{BCCS} for a rigorous
construction of the model based on an equivalent change of probability measure insuring the well definition of the model. }.
Let $(z_i)_{i\in \N}$  be the independent identically distributed marks with distribution $\pi$, a positive square integrable function describing the law of rainfall during storms, its domain will be denoted by $\Pi\subseteq \mathbb{R}^{+}\setminus{\{0\}}$.
The sequence of random jump times $(\Theta_i)_{i\in \N}$ represents the arrival of a rainfall whereas the marks $(z_i)_{i\in \N}$ are associated with the impact on the level of the dam. 
The process $Z$ is a non-decreasing  marked point process 
$$
Z_t=\Sum_{i=1}^{J_t} z_i = \Sum_{i\in \N } z_i \1_{i\leq J_t} 
$$ adding the marks up to time $t$. 

The controlled water level dynamics is then given by the following stochastic differential equation for the pair $(H,\lambda)$ 
\begin{equation} \label{H-dynamics}
 \left\{
 \begin{array}{l}
 dH_t =  -I_t \cdot \varphi(H_t) dt - \nu_t dt +dZ_t,  \\
 d\lambda_t =  a (b-\lambda_t) dt + c \,  dZ_t,  \\
  H_o=h\in [0, h^{max} ]\\
   \lambda_0=\ell \in [b,+\infty)
 \end{array}
 \right.
 \end{equation}
 where $I_t$ denotes the turbine status at time $t$ and it is assumed to be a control variable; it can take values in $\{0,1\}$,  $I_t=1$ when the turbine is operating at time $t$, while $I_t=0$ describes a closed turbine at time $t$. The second control variable $\nu_t $ denotes the flow of water transiting  spillways status. The compensator of the compounded jump process $Z_t$ reads $\lambda_t \pi (dz) dt$. 
In particular, when the turbine and the spillway are closed,  first SDE in \eqref{H-dynamics} reads $ dH_t = dZ_t$.

We point out that the probability to have a new storm increases due to the self-exciting property of the Hawkes processes. 
The three parameters $a$, $b$ and $c$ are positive: $a$ represents the speed reversion to the long-term mean value $b$ and $c$ the self exciting effect of each rainfall.

As in \cite{ChDiGaGiSc}, the deterministic function $\varphi$ appearing in \eqref{H-dynamics} can be determined by physical arguments: it is given by the product of the open section $Q_t>0$ of the penstock connecting the basin with the turbine and the speed $v_t$ of the water at the entrance of the conduct divided by
the surface $S$ of the basin assumed to be independent from the height $h$. The assumption of surface $S$ constant is motivated by considering an abstract class of dams for which the bassin $S$ can be supposed constant.  Indeed that speed $v_t$ can be expressed as a function of the height by the energy conservation principle; we can then write:
\begin{equation}
\varphi(H_t)=Q_t\frac{v_t}{S} =Q_t\frac {\sqrt{2g(H_t-h_0)}}{S}
\end{equation}
where $g$ denotes the gravity constant.
The power produced, in terms of energy units, is given by the following expression:
\begin{equation}	\label{eq3.2}
\mathcal E_t= Q_t v_t g (H_t-h_0 ) (1- \chi)=  \varphi(H_t) S g (H_t-h_0 ) (1- \chi) = \sqrt{2} (1- \chi) Q_t  \left[g (H_t-h_0 )\right]^{3/2},
\end{equation}
where $\chi$ is a dispersion coefficient.
The conduct connecting the reservoir with the turbine is designed for the dam and a crucial issue for power production is to keep the frequency of the alternate electric current constant (generally 50Hz in European Countries,  60Hz in American Countries).   Since this is determined by the rotational speed of the turbine, the open section area will be regulated by servo-mechanism in such a way that water flowing per second $Q_t$ will keep constant the rotational frequency of the turbine. That implies that $\mathcal E_t$ could be considered constant as well and in the sequel it will be simply denoted by $\mathcal E$, representing the power per unit of time produced by the turbine when operates.

At the same time we obtain that the volume of water flowing $Q_t$ as inversely proportional to the head of the water power $3/2$ as we seen in \eqref{eq3.2}.
We can then write the effect of power production on the water level as follows:
\begin{equation}
\varphi(H_t)=\frac{Q_t \sqrt{2g (H_t-h_0)}}{S} = \frac{\mathcal E }{S g(1- \chi)} \;\; \frac{1}{H_t-h_0 }.
\end{equation}
That is the impact of the electricity production on the head of water is inversely proportional to itself.

By a similar argument the spillover impact $\nu_t$ is proportional to the section of spillover pipe and the square root of the free heigh $(H_t-h_0)$, where we suppose that the free heigh is the same for spillover and turbine pipes, as in Hoover dam. As a consequence, 
the spillover impact can be rewritten as $\beta_t \sqrt{2g (H_t-h_0)}$, where $\beta_t$  does not depend of $H_t$ and  takes values in the interval $[\beta^{min}, \beta^{max}]$, where $\beta^{min}\geq 0$ will be specified in \eqref{contraint-beta}  and $\beta^{max}>\beta^{min}$ describes the spillways opening with maximum flow, which depends on physical constraints of the dam.

We believe that the evolution of the electricity price is not necessary since, being always positive, we always could reduce of the dimension by the numeraire technique as done in \cite{ChDiGaGiSc} . In this way, the problem could be rewritten under a new probability measure equivalent to the historical one under which the price of the electricity is constant, see the seminal paper by Geman et al. \cite{Geman} and the dimension of the state variable decreases by one. The change of probability has been explicitly detailed in the dam management problem in the paper by Chevalier et al. \cite[Section 3.1]{ChDiGaGiSc}.

\section{The dam constraints and the optimal control problem} \label{The Optimal Control Problem}

We formulate our dam management problem as an optimal stochastic control problem. The goal of the dam manager is to maximise the revenues from the power production by keeping the safety level of the dam under control, moreover other external constraints are added.
The reservoir of a dam is not only used to produce electricity but also to guarantee other relevant requests in particular we point to the following usual requests:
\begin{description}
\item[Dam protection:] to keep the water level under control for safety reasons, a penalty is introduced when the water level overcomes the critical value $h_{+} $.
The penalty function, with support $\lbrack  h_{+} , h^{max}\rbrack$, is denoted by $f_+$ and Lipschitz continuous on the same interval.

\item[Touristic level:] the reservoir behind the dam is often used for touristic purposes. If the heigh of the basin is too low there is an impact on touristic use and then 
a penalty is introduced when the water level falls below the critical value $h_{-} $, known in France as  ``Cote Touristique" meaning Touristic level.
The penalty running function, defined on $\lbrack 0,  h_{-} \rbrack$, is denoted by $f_-$ and Lipschitz continuous  on the same interval.

\item[Low-flow period:] during the long periods of drought, the dam contributes to guarantee a minimal flow  $\mu$ in order to preserve biodiversity and to support agriculture.
Since the intensity $\lambda$ also measures, in an indirect way, the lack of rains and then the drought, we add a constraint that a minimal outflow $\mu$ has to be fulfilled (given by the turbine penstock or the spillover) if the rains' intensity $\lambda$ is lower than a threshold $\underline{\ell}$.  
 
\end{description}

Let now $(\tau_k)_{k\geq 1}$ be a non decreasing sequence of $\mathbb F$-stopping times such that $\lim_{k\rightarrow +\infty}\tau_k=+\infty$. At the random and controlled  time $\tau_k$, the dam manager decides to switch the turbine status from state $0$ (no production) to state $1$ (operating turbine) or vice-versa. 
We recall that the sequence  $(\Theta_i)_{i\in\N}$ of random jump times are totally inaccessible. On the contrary random control times $(\tau_k)_{k\geq 1}$ are predictable. 
See the definitions in Jacod and Shiryaev, \cite[Chapter I]{JacShi2003} for totally inaccessible jump times.  We associate a controlled Markov chain taking values in $\{0,1\}$ and denoted by $I$ and, for $t\geq 0$, by
\begin{equation}	\label{It}
 I_0=i\in \{0,1\} \quad \mbox{ and }  \quad I_t :=\frac{1}{2} \Big[  1 - (-1)^{ I_0 + N_t } \Big] , \quad 
\mbox{ where } \quad N_t := \sum_{k=1}^{\infty} \mathbbm{1}_{ \tau_{k} \leq t }.
\end{equation}

\begin{Remark}	\label{rem3.1}
Two sequences of random time $\{\Theta_i \}_i$ and $\{\tau_i\}_i$ represents the rains times and the sequence the switching times of the turbine decided by the manager. Those times are completely different by the nature of the problem. Mathematically times $\{\Theta_i\}$ are totally  inaccessible, times $\{\tau_i\}_i$ are predictable and it holds for those times $\{\tau_i \}_i \cap \{\Theta_i\}_i=\emptyset$ a.s..  As \cite{JacShi2003} says at page 20 ``\emph{Now we introduce a class of stopping times that are ``orthogonal" in a sense to all predictable times.}''. 
\end{Remark}

The second control variable $\beta$ is an $\mathbb F$-predictable process.
During low-flow period, we introduce in the model the minimal flow $\mu$ in order to preserve biological or agricultural needs. Since the cross-sectional area $S$ of the artificial lake is assumed to be constant, it follows that the decrease in the water level per unit time inside the reservoir is directly proportional to the water flow released into the downstream river\footnote{The parameter $\mu$ represents a minimum admissible rate of decrease of the reservoir level. Throughout the paper, for simplicity, the reservoir surface is normalised to one. Under this normalisation, $\mu$ is expressed in units of water-level variation per unit time rather than as a volumetric discharge. For a reservoir with physical surface area $S$ the corresponding minimum environmental release would be $Q_{\min}=\mu \cdot S$, where $Q_{\min}$ is measured in volume per unit time. Therefore, $\mu$ should be interpreted as a normalised proxy for environmental-flow requirements, while the actual discharge constraint depends on the physical size of the reservoir. The approximation $Q_{\min}=\mu \cdot S$ is exact when the reservoir surface is assumed constant and should be interpreted as a first-order approximation when the surface depends on the water level.} Requiring a guaranteed minimum outflow is therefore equivalent to controlling a minimum rate of decrease of the water level inside the dam in the absence of rainfall. 

\begin{Remark}
Throughout the paper, the reservoir cross-sectional area is assumed to be constant and normalized to one. Under this convention, the volumetric discharge rate and the corresponding rate of decrease of the reservoir level are numerically identical. Hence the minimum flow parameter $\mu$ may equivalently be interpreted as a minimum outflow requirement or as a minimum rate of decrease of the water level, although the two quantities have different physical dimensions before normalization.
\end{Remark}

We then have

\begin{equation}\label{contraint-beta}
\beta_t \geq \mathbbm{1}_{\lambda_{t-} \leq \underline{\ell} } \max \left\{ \mu -  \varphi(H_{t-}) \mathbbm{1}_{I_{t-} =1 } \, , \, 0 \right\} =:\beta^{min}. 
\end{equation}
We denote by $\mathcal B$ the set of such processes.\\

The problem will be formulated as a stochastic control problem mixing regular and switching controls.
We introduce the following set of admissible controls:
\beq
\mathcal A:=\left\lbrace \alpha=(\beta, (\tau_k)_{k\geq 1}):\ \beta\in\mathcal B\textrm{ and } (\tau_k)_{k\geq 1}\textrm{ is a non decreasing sequence of }\F-\textrm{stopping times}\right\rbrace.
\enq

We denote by $H^\alpha$ the solution of the first controlled SDE in \eqref{H-dynamics} for a given control $\alpha$. When no ambiguity can arise, we will denote $H^\alpha$ by $H$.
When the water level reaches the value $h^{max} > h_{+}$, the maximum allowed by safety reasons, the entire problem ends,
the optimisation problem is terminated and a terminal penalty $P$ is incurred.
This problem is usually called overtopping failure and represents one of the most frequent reason of dam failure, see Tingsanchali and Chinnarasri \cite{Tingsanchali}. Usually overtopping is a consequence of heavy rains as in the case of Laurel Run Dam failure in 1977.
We can then define the time horizon $T^{\alpha}$ of our stochastic control problem as
$ T^{\alpha}=\inf \left\{t>0 \, \vert \,  H^{\alpha}_t\geq h^{max}\right\} $, when no ambiguity can arise, we will denote  it by $T$.

This optimization problem depends naturally on two variables: the water level and the intensity of the Hawkes process describing rainfall probabilities. We focus ont the dry effects represented by the Hawkes process in the proposed model. So we have two value functions associated to the problem and denoted respectively by $v_0$, if the initial status of the turbine is closed, and $v_1$, if the initial status of the turbine is open. 
Both value functions are defined on $D:=[0, h^{max} ]\times [b,+\infty)$, by the following expression:
\begin{equation}\label{eqn-value-function}
\begin{split}
 v_i(h,\ell)& := \sup_{\alpha \in \mathcal{A}} \, \mathbb{E}^{i,h, \ell}\Bigg[ \int_0^{T-} e^{- \rho t} \mathcal E \, I_t  dt
- \int_0^{T-} e^{- \rho t} \kappa \; dN_t  \\
& \quad\quad\quad\quad\quad \quad\quad\quad -  \int_0^{T-} e^{-  \rho t}  \; \Bigg( f_{+}\left(\left(H_t- h_{+}\right)^{+}\right)  + f_{-}\left(\left(h_{-} - H_t\right)^{+}\right) \Bigg) dt -Pe^{- \rho T} \Bigg],
\end{split}
\end{equation}
where $(h,l)=(H_0,\lambda_0)$ in \eqref{H-dynamics}, 
 $I_t $ is defined in \eqref{It}, i.e. $i=I_0$, $\rho$ is the discount factor, $\kappa$
denotes the cost, in electricity, of switching from operating to a closed turbine and $P$ is a fixed cost associated with dam failure.  $\mathbb{E}^{i,h,\ell}$ denotes the expectation 
$\mathbb{E}$ and stress its dependance by values $i,h, \ell$. This is a classical and widely accepted notation in the fields of optimal control.\\

To solve our optimisation problem, we will exploit the following dynamic programming principle
that we write in its classical version for sake of readability. In this context we do not have a priori regularity results on the value function, in particular we do not even know if the value function is continuous. Mainly motivated by not-enough regular value functions, \cite{BouchTouzi2011} proved a so-called weak dynamic programming principle (DPP) in the diffusion case (no jumps). 
They established a sub-(resp., super-) optimality principle of dynamic programming involving its upper-(resp., lower-) semicontinuous envelope $v^*$ (resp. $v_*$) for a value function $v$. 
The diffusive case has been largely studied, we refer to the book \cite[Chapter 3.2]{Touzi}. 
We need a weak formulation of the dynamic programming which avoids the heavy measurable selection arguments in typical proofs of the dynamic
programming principle when no a priori regularity of the value function is known as in this case. 
At our knowledge, the DPP in controlled jump diffusions, as already pointed out in the
introduction of \cite{Ishi2004}, the literature is limited. Some steps forward this direction are currently under work, as in \cite{Harv} or more recently \cite{BondiPriola25}.
The proof is not trivial but it is quite reasonable to admit the following, a detailed proof, that can be adapted to our setup, in the case of Hawkes process is done in the appendix of \cite{Harv} for the more tricky case of singular control.

\begin{Proposition}	\label{DPP}
For all $(h,\ell)\in D$ and any $\F$-stopping times $\theta$, we have
\begin{equation}\label{DPPeq}
\begin{array}{rcr}
 v_i(h,\ell)& =& \displaystyle \sup_{\alpha\in  \mathcal{A}} \, \mathbb{E}^{i,h, \ell}\Bigg[ \int_0^{(\theta\wedge T)} e^{- \rho t} \; \Bigg( \mathcal E \, I_t  -    f_{+}\left(\left(H_t- h_{+}\right)^{+}\right)  - f_{-}\left(\left(h_{-} - H_t\right)^{+}\right) \Bigg) dt \quad 
  \\
&&\displaystyle \quad \quad \quad \quad \quad \quad - \int_0^{(\theta\wedge T)-} e^{- \rho t} \kappa \; dN_t +
e^{-\rho \theta} v_{I_{\theta}} (H_\theta,\lambda_\theta)\mathbbm{1}_{\theta<T} -Pe^{- \rho T} \mathbbm{1}_{\theta\geq T}\Bigg]\\
& =& \displaystyle\sup_{\alpha\in  \mathcal{A}} \, \mathbb{E}^{i,h, \ell}\Bigg[ \int_0^{(\theta\wedge T)} e^{- \rho t} \; \Bigg( \mathcal E \, I_t  -    f_{+}\left(\left(H_t- h_{+}\right)^{+}\right)  - f_{-}\left(\left(h_{-} - H_t\right)^{+}\right) \Bigg) dt \quad   \\
&& \displaystyle \quad\quad \quad \quad \quad \quad  - \int_0^{(\theta\wedge T)-} e^{- \rho t} \kappa \; dN_t + e^{-\rho (\theta\wedge T)} v_{I_{\theta\wedge T}}(H_{\theta\wedge T},\lambda_{\theta\wedge T}) \Bigg].
\end{array}
\end{equation}
\end{Proposition}
%

\begin{Remark}[Dynamic programming for jump processes and Hawkes-type dynamics]

The dynamic programming principle (DPP) for controlled diffusion processes is by now classical.
Under standard measurability and growth assumptions, the value function of a Markov diffusion
satisfies a strong DPP with respect to all stopping times and is characterized as the unique
viscosity solution of the associated Hamilton--Jacobi--Bellman (HJB) equation; see, for instance,
the monographs \cite{FlemingSoner2006,BardiCapuzzo1997,YongZhou1999,Pham2009} and the references therein.

For controlled jump processes the picture is more delicate. For abstract jump Markov and
piecewise deterministic Markov processes one already obtains nonlocal HJB equations in the
viscosity sense, see Soner \cite{Soner1988}. In the case of SDEs driven by L\'evy noise, weak
forms of the DPP, sufficient for the viscosity characterization of integro--differential HJB
equations, are developed in the general framework of Bouchard and Touzi \cite{BouchTouzi2011}
and in the jump--diffusion literature surveyed in \cite{OksendalSulem2019}. Strong DPP results,
where concatenation at stopping times is proved for controlled jump diffusions with both small
and large jumps, require a careful analysis of stochastic flows in the Skorokhod topology and
subtle measurable selection arguments; see in particular the recent work by Bondi and Priola
\cite{BondiPriola25} and also \cite{GoldysWu2016}. Compared with the diffusion case, the
discontinuity of the trajectories and the nonlocal generator make stability with respect to the
control, approximation by piecewise-constant controls and the construction of a regular flow
substantially harder.

In our setting the controlled dynamics combine a diffusion component with a Hawkes-type
self-exciting point process. In the Markovian exponential-kernel case, by augmenting the state
with the current intensity, one can represent the pair $(N_t,\lambda_t)$ as a Markov process,
and it is natural to expect a DPP of the usual form, in analogy with the L\'evy-driven models
treated in \cite{OksendalSulem2019,BouchTouzi2011,BondiPriola25,GoldysWu2016}. However, to the best of our
knowledge there is no general theorem covering controlled Hawkes processes with state- and
history-dependent intensities at the level of generality considered here. For this reason, and
in view of the probabilistic difficulties mentioned above, we shall assume in
Proposition \ref{Prop4.1} and Theorem \ref{THMValue} that a classical DPP holds for our Hawkes-driven control problem. This
assumption is consistent with the Markov structure of the augmented state and with the available
results for jump diffusions, and it is only used to justify the HJB system which characterizes
the value function.
\end{Remark}

From the above dynamic programming principle, we may derive the HJB equation associated with the optimal switching control problem \eqref{eqn-value-function}.
The proposed variational inequality corresponding to the optimal stochastic control is the following
\begin{equation}\label{HJB}
\min \Bigg\{  \rho v_i  - \sup_{\beta  \in\mathcal B} { \mathcal{L}^{(i,\beta)} v_i} +  f_{+} \left(\left(h- h_{+}\right)^{+}\right) +  f_{-} \left(\left(h_{-}- h\right)^{+}\right)  -i \mathcal E \; ; \; \left(v_{i} - v_{ 1-i } - \kappa \right) \Bigg\} =0,\quad\textrm{ on }\mathring D,
\end{equation}
where the associated integral first order differential operator, called Lagrangian $  \mathcal{L}^{(i,\beta)} $, is defined as follows:
\begin{equation} \label{operator-L}
\begin{split}
\mathcal{L}^{(i,\beta )} \phi(h,\ell) =
  \left [ - i  \frac{\mathcal E }{S g(1- \chi)} \frac{1}{h - h_0 }
  - \beta\sqrt{2 g (h - h_0 )^+} \right] \, \frac{\partial \phi}{\partial h}(h,\ell)
   \quad \quad \quad   \quad\quad  \\
 \quad  + a(b- \ell)\, \frac{\partial \phi}{\partial \ell}(h,\ell) + \ell \int_{\Pi} \Big[ \phi(h+z,\ell+cz) - \phi(h,\ell) \Big] \pi(dz).
\end{split}
\end{equation}
Optimizing the free control $\beta$ according with the constraint \eqref{contraint-beta}, we obtain 
\begin{equation} \label{operator-Lsup}
\begin{array}{rcl}
\displaystyle \sup_{\beta \in \mathcal{B} }\mathcal{L}^{(i,\beta )} \phi(h,\ell) 
&=& \displaystyle - i \frac{\mathcal E }{Sg (1- \chi )(h - h_0)}   \frac{\partial \phi}{\partial h} 
  + \beta^{max} \, \sqrt{2g( h - h_0)^+ }
 \left(  \frac{\partial \phi}{\partial h} \right)^- \\
 &&\displaystyle +a(b- \ell)\, \frac{\partial \phi}{\partial \ell} 
 +  \ell \int_{\Pi} \Big[ \phi(h+z,\ell+cz) - \phi(h,\ell) \Big] \pi(dz)  
  \\
  & &\displaystyle  -    \left( \mu -  i \frac{\mathcal E }{S g(1- \chi)} \;\; \frac{1}{h-h_0 }  \right)^+ \;  \left(  \frac{\partial \phi}{\partial h} \right)^+ \; \mathbbm{1}_{\ell \leq \underline{\ell} }   \ .
\end{array}
\end{equation}

\section{Characterization as Viscosity Solution}  \label{optimal-switching}

In the present section we characterize the solution of the problem formulated in the previous section as a viscosity solution of the following HJB equation
\begin{equation} \label{eq-HJB}
\begin{array}{rcl}
0 & =& \displaystyle \min \left\{  \rho v_i  + i \frac{\mathcal E }{Sg (1- \chi )(h - h_0)}   \frac{\partial v_i}{\partial h} - \beta^{max} \, \sqrt{2g( h - h_0)^+ }
 \left(  \frac{\partial v_i}{\partial h} \right)^- \right.  \\
&& \displaystyle \quad\quad\quad
 - a(b- \ell)\, \frac{\partial v_i}{\partial \ell} -  \ell \int_{\Pi} \Big[ v_i(h+z,\ell+cz) - v_i(h,\ell) \Big] \pi(dz)  \\
&& \displaystyle \quad\quad\quad  +   \left( \mu -  i \frac{\mathcal E }{S g(1- \chi)} \;\; \frac{1}{h-h_0 } \right)^+ \;  \left(  \frac{\partial v_i}{\partial h} \right)^+ \;
\mathbbm{1}_{\ell \leq \underline{\ell} }
\\
 && \displaystyle \quad\quad\quad + f_+ \left((h- h_{+})^{+}\right) + f_- \left((h_- - h)^{+}\right) - i \mathcal E \; ; \; v_i - v_{1-i} + \kappa \Big\}
 \end{array}
\end{equation}
in the interior of $D$. We also have the boundary condition $v_i(h^{max}, \ell)=-P$.\\
%
%
%
We start with a standard comparison result, which says that any smooth function, which is a supersolution to the HJB equation \eqref{eq-HJB}, dominates $v_i$.
\begin{Proposition} 	\label{Prop4.1} 
Let $i\in \{0,1\}$, $\phi_i \in C^1(D;\R)$ be such that  $\phi_i (h^{max},\ell)\geq -P$ and 
supersolution of \eqref{eq-HJB} on $D$, i.e.
\begin{equation}		\label{Eq34}
\begin{split}
& \min \left\{  \rho \phi_i(h,l)  + i \frac{\mathcal E }{Sg (1- \chi )(h - h_0)}   \frac{\partial \phi_i}{\partial h} (h,l) - \beta^{max} \, \sqrt{2g( h - h_0)^+ }
 \left(  \frac{\partial \phi_i}{\partial h}(h,l)  \right)^- \right. \\
& \quad\quad\quad
 - a(b- \ell)\, \frac{\partial \phi_i}{\partial \ell}(h,l)  -  \ell \int_{\Pi} \Big[ \phi_i(h+z,\ell+cz) - \phi_i(h,\ell) \Big] \pi(dz)  \\
& \quad\quad\quad  +  \left( \mu -  i \frac{\mathcal E }{S g(1- \chi)} \;\; \frac{1}{h-h_0 } \right)^+ \; \left(  \frac{\partial \phi_i }{\partial h} (h,l)  \right)^+ \;  \mathbbm{1}_{\ell \leq \underline{\ell} }
+ f_+ \left((h- h_{+})^{+}\right) 
\\
 & \quad\quad\quad + f_- \left((h_- - h)^{+}\right) - i \mathcal E \; ; \; \phi_i (h,l) - \phi_{1-i}(h,l)  + \kappa \Big\} \geq 0  \quad \forall (h,\ell)\in D\ .
 \end{split}
\end{equation}
Then we have $\phi_i(h,\ell)\geq v_i(h, \ell)$ for all $(h,\ell)\in D$.
\end{Proposition}
\begin{proof}
Let consider an initial state-regime value $(h,\ell;i)\in D\times \{ 0,1\}$. 
Take an arbitrary control $\alpha=\left( (\beta_t)_{t\geq 0}, \{ \tau_n \}_{n\in \N} \right)$. We set $$\tilde \tau_{j}:=\tau_j \wedge T\wedge 
\inf\left\{t\geq 0:\, \max\{H_t,\lambda_t\} \geq j \right\} 
$$ for any $j\in \N$. 
We then apply a change of variables formula and It\^o formula in \cite[Chapter II, Theorems 31,33]{Protter} to $e^{-\rho t} \phi_{I_t}\left(H^{(h,\ell;i)}_t, \lambda^{(h,\ell;i)}_t\right)$ for càd-làg semimartingales between finite stopping times $\tilde \tau_n$ and $\tilde \tau_{n+1}$. 
For the sake of simplicity reason we will denote $\left(H^{(h,\ell;i)}_t, \lambda^{(h,\ell;i)}_t\right)$ simply by $\left(H_t, \lambda_t\right)$. 
Applying between $\tilde\tau_0=0$ and $\tilde \tau_1^{-}$ we get 
\[
\begin{split}
e^{-\rho \tilde\tau_{1}}\phi_{i}\left(H_{\tilde \tau_{1}^-}, \lambda_{\tilde \tau_{1}^-} \right)=& \phi_{i}\left(h, \ell \right)+\int_{0}^{\tilde \tau_{1}} e^{-\rho t} \left( \mathcal{L}^{(i,\beta )} \phi_{i}-\rho \phi_{i}\right)\left( H_{t^-}, \lambda_{t^-} \right) dt\\
&-\int_{0}^{\tilde \tau_1} e^{-\rho t} \lambda_{t^-} \int_{\Pi} \left[ \phi_i\left(H_{t^-}+z,\lambda_{t^-}+cz\right) - \phi_i \left( H_{t^-},\lambda_{t^-}\right)\right] \pi(dz) \, dt\\
&+\sum_{0\leq t< \tilde\tau_1} e^{-\rho t} 
\left\{\phi_i\left(H_t,\lambda_t \right) -\phi_i\left(H_{t^-}, \lambda_{t^-} \right) 
\right\}
\end{split}
\]
Analogously between $\tilde\tau_n$ and $\tilde\tau_{n+1}^{-}$ we get
\begin{equation}	\label{tre}
\begin{split}
e^{-\rho \tilde\tau_{n+1}}
\phi_{I_{(\tilde\tau_{n+1})^-}  }   
\left( H_{\tilde \tau_{n+1}-}, \lambda_{\tilde \tau_{n+1}-} \right)
&
=e^{-\rho \tilde\tau_{n}} \phi_{I_{\tilde\tau_n}}\left(H_{\tilde\tau_n}, \lambda_{\tilde\tau_n} \right)+\int_{\tilde \tau_n}^{\tilde \tau_{n+1}} e^{-\rho t} \left( \mathcal{L}^{(I_{\tilde\tau_n},\beta )} \phi_{I_{\tilde\tau_n}}-\rho \phi_{I_{\tilde\tau_n}}\right)
\left(H_{t^-},\lambda_{t^-}\right) dt\\
&-\int_{\tilde\tau_n}^{\tilde \tau_{n+1}} e^{-\rho t} \lambda_{t^-} \int_{\Pi} \left[ \phi_{I_{\tilde\tau_n}}\left(H_{t^-}+z,\lambda_{t^-}+cz\right) - \phi_{I_{\tilde\tau_n}}\left( H_{t^-},\lambda_{t^-}\right)\right] \pi(dz) \, dt\\
&+\sum_{\tilde\tau_n \leq t< \tilde\tau_{n+1}} e^{-\rho t} 
\left\{\phi_{I_{\tilde\tau_n}}\left(H_t,\lambda_t \right) -\phi_{I_{\tilde\tau_n}}\left(H_{t^-}, \lambda_{t^-} \right) 
\right\}
\end{split}
\end{equation}
Considering that $H_t, \lambda_t$ have same jump times, it is possible to rewrite the last two terms as an integral with respect to the compensated process, i.e.
\begin{equation}	\label{trE}
\begin{split}
&\sum_{\tilde\tau_n\leq  t< \tilde\tau_{n+1}} 
\left\{\phi_{I_{\tilde\tau_n}}\left(H_t,\lambda_t \right) -\phi_{I_{\tilde\tau_n}}\left(H_{t^-}, \lambda_{t^-} \right) \right\} e^{-\rho t}  \\
&\quad \quad \quad -\int_{\tilde\tau_n}^{\tilde \tau_{n+1}} e^{-\rho t} \lambda_{t^-} \int_{\Pi} \left[ \phi_{I_{\tilde\tau_n}}\left(H_{t^-}+z,\lambda_{t^-}+cz\right) - \phi_{I_{\tilde\tau_n}}\left( H_{t^-},\lambda_{t^-}\right)\right] \pi(dz) \, dt
\\
&= \int_{\tilde\tau_n}^{ \tilde\tau_{n+1}}  \left\{\phi_{I_{\tilde\tau_n}}\left(H_t,\lambda_t \right) -\phi_{I_{\tilde\tau_n}}\left(H_{t^-}, \lambda_{t^-} \right) \right\} e^{-\rho t} d\widetilde J_t
\end{split}
\end{equation}
where $\{\widetilde{J}_t\}_{t\geq 0}:=\{J_t-\int_0^t\lambda_s ds\}_{t\geq 0}$ is the compensated counting process, i.e. the  $(\mathbb{P},\mathbb{F})$-local martingale associated to $J$. 
The integrands are bounded due to the localisation on $[{\tilde\tau_n}, {\tilde\tau_{n+1}})$, so $\widetilde J$ is a true martingale with zero mean. Finally replacing \eqref{trE} in relation \eqref{tre} and taking the expectation we get 
\begin{equation}	\label{eq33}
\begin{split}
\mathbb{E}\left[  e^{-\rho \tilde\tau_{n+1}}
\phi_{I_{\tilde\tau^-_{n+1}}  }   
\left( H_{\tilde \tau^-_{n+1}}, \lambda_{\tilde \tau^-_{n+1}} \right)
 \right]
 = &\mathbb{E}\left[  e^{-\rho \tilde\tau_{n}} \phi_{I_{\tilde\tau_n}}\left(H_{\tilde\tau_n}, \lambda_{\tilde\tau_n} \right) \right]\\
 &+\mathbb{E}\left[ \int_{\tilde \tau_n}^{\tilde \tau_{n+1}} e^{-\rho t} \left( \mathcal{L}^{(I_{\tilde\tau_n},\beta )} \phi_{I_{\tilde\tau_n}}-\rho \phi_{I_{\tilde\tau_n}}\right)\left( H_{t^-}, \lambda_{t^-} \right) dt \right]\\
 &+\mathbb{E}\left[  \int_{\tilde\tau_n}^{ \tilde\tau_{n+1}}  \left\{\phi_{I_{\tilde\tau_n}}\left(H_t,\lambda_t \right) -\phi_{I_{\tilde\tau_n}}\left(H_{t^-}, \lambda_{t^-} \right) \right\} e^{-\rho t} d\widetilde J_t \right]\\
 =&\mathbb{E}\left[ e^{-\rho \tilde\tau_{n}} \phi_{I_{\tilde\tau_n}}\left(H_{\tilde\tau_n}, \lambda_{\tilde\tau_n} \right) \right]\\
 &+\mathbb{E}\left[ \int_{\tilde \tau_n}^{\tilde \tau_{n+1}} e^{-\rho t} \left( \mathcal{L}^{(I_{\tilde\tau_n},\beta )} \phi_{I_{\tilde\tau_n}}-\rho \phi_{I_{\tilde\tau_n}}\right)\left( H_{t^-}, \lambda_{t^-} \right) dt \right]\\
\end{split}
\end{equation}
By \eqref{Eq34} we obtain
\begin{equation*}
\left(
\mathcal{L}^{(I_{\tilde\tau_n},\beta )} 
\phi_{I_{\tilde\tau_n}}-\rho \phi_{I_{\tilde\tau_n} }\right)\left( H_{t^-}, \lambda_{t^-} \right) 
\leq -I_{\tilde\tau_n}\mathcal{E}+ f_+ \left((H_t- h_{+})^{+}\right) + f_- \left((h_- - H_t)^{+}\right) .
\end{equation*}
Then \eqref{eq33} becomes
\begin{equation*}
\begin{split}
\mathbb{E}\left[  e^{-\rho \tilde\tau_{n+1}}
\phi_{I_{\tilde\tau^-_{n+1}}  }   
\left( H_{\tilde \tau^-_{n+1}}, \lambda_{\tilde \tau^-_{n+1}} \right)
 \right]
 \leq & \mathbb{E}\left[ e^{-\rho \tilde\tau_{n}} \phi_{I_{\tilde\tau_n}}\left(H_{\tilde\tau_n}, \lambda_{\tilde\tau_n} \right) \right]\\
 &-\mathbb{E}\left[  \int_{\tilde \tau_n}^{\tilde \tau_{n+1}} e^{-\rho t} \left\{ \mathcal{E}I_{t}- f_+ \left((H_t- h_{+})^{+}\right) - f_- \left((h_- - H_t)^{+}\right)\right\} dt  \right]\\
  \end{split}
\end{equation*}
which is equivalent to
\begin{equation}	\label{eq35}
\begin{split}
\mathbb{E}\left[  e^{-\rho \tilde\tau_{n}}\phi_{I_{\tilde\tau_n}}\left(H_{\tilde\tau_n}, \lambda_{\tilde\tau_n} \right) \right]
\geq &
\mathbb{E}\left[  e^{-\rho \tilde\tau_{n+1}}
\phi_{I_{\tilde\tau_{n+1}^-}  }   
\left( H_{\tilde \tau_{n+1}^-}, \lambda_{\tilde \tau_{n+1}^-} \right)
 \right]\\
 &+
\mathbb{E}\left[  \int_{\tilde \tau_n}^{\tilde \tau_{n+1}} e^{-\rho t} \left\{ \mathcal{E}I_{t}- f_+ \left((H_t- h_{+})^{+}\right) - f_- \left((h_- - H_t)^{+}\right)\right\} dt \right] \ .
\end{split}
\end{equation}		
Recalling that $H$ and $\lambda$ has the same jump times and exploiting the second term in super solution inequality \eqref{Eq34} we obtain 
\begin{equation}	\label{eq2222}
\phi_{I_{\tilde\tau_{n+1}^-}  }   
\left( H_{\tilde \tau_{n+1}^-}, \lambda_{\tilde \tau_{n+1}^-} \right)
\geq 
\phi_{I_{\tilde\tau_{n+1}}}\left(H_{\tilde\tau_{n+1}^-} , \lambda_{\tilde\tau_{n+1}-}  \right) -k=\phi_{I_{\tilde\tau_{n+1}}}\left(H_{\tilde\tau_{n+1}} , \lambda_{\tilde\tau_{n+1}}  \right) -k
\end{equation}
where the equality in \eqref{eq2222} follows by the fact that  $(\tau_{i})_{i \in \N} \cap (\Theta_{i})_{i\in \N}=\emptyset$ a.s., see Remark \ref{rem3.1}.
Taking into account \eqref{eq2222}, \eqref{eq35} gives
\begin{equation*}
\begin{split}
\mathbb{E}\left[  e^{-\rho \tilde\tau_{n}}\phi_{I_{\tilde\tau_n}}\left(H_{\tilde\tau_n}, \lambda_{\tilde\tau_n} \right) \right]
&\geq 
\mathbb{E}\left[  e^{-\rho \tilde\tau_{n+1}} 
\left( \phi_{I_{\tilde\tau_{n+1}}}\left(H_{\tilde\tau_{n+1}} , \lambda_{\tilde\tau_{n+1}} \right) -k \right)
 \right]\\
 &+
\mathbb{E}\left[  \int_{\tilde \tau_n}^{\tilde \tau_{n+1}} e^{-\rho t} \left\{ \mathcal{E}I_{t}- f_+ \left((H_t- h_{+})^{+}\right) - f_- \left((h_- - H_t)^{+}\right)\right\} dt \right] .
\end{split}
\end{equation*}
By iterating the previous inequality from $0$ for all $n$ up to $T$ and recalling that $H_T=h^{max}$ we then obtain
\begin{equation*}
\begin{split}
\phi_i(h,\ell)&\geq 
\mathbb{E}\left[ e^{-\rho T}\phi_{I_T}(H_T,\lambda_T) -k\sum_{\tilde\tau_n\leq T}e^{-\rho \tilde\tau_n}+\int_0^T e^{-\rho t}\left\{ \mathcal{E}I_t - f_+ \left((H_t- h_{+})^{+}\right) - f_- \left((h_- - H_t)^{+}\right)\right\}  dt \right]\\
&
\geq \mathbb{E}\left[ \int_0^T e^{-\rho t}\left\{ \mathcal{E}I_t - f_+ \left((H_t- h_{+})^{+}\right) - f_- \left((h_- - H_t)^{+}\right)\right\}  dt    -\int_{0}^{T^-} ke^{-\rho t }dN_t  -Pe^{-\rho T} \right] .
\end{split}
\end{equation*}
By the arbitrariness of the control  $\{\{ \beta_t \}_{t\geq 0}, \{\tau_k\}_{k\in\N}\} $ we obtain the required result.
\end{proof}

\subsection{Existence of solutions} 

We have the following PIDE characterization of the value functions $v_i$

\begin{Theorem}	\label{THMValue}
The value functions $v_i$ are viscosity solutions on $\mathring D$ to \eqref{eq-HJB} such that $v_{i}(h^{max}, \ell)=-P$.
\end{Theorem}
\begin{proof}[Proof of super solution property.] Fix $i\in \{1,0\}$, let $(\hat h, \hat \ell)\in \mathring D$ and $\phi_i\in C^1(D,\R)$ such that $(\hat h, \hat \ell)$ is a minimum of $v_i-\phi_i$ with  $\phi_i(\hat h, \hat \ell)=v_i(\hat h, \hat \ell)$. Of course $(v_i-\phi_i)(h,\ell)\geq 0$ on $D$. 
We have to prove that 
\begin{equation}
\begin{split}
& \min \left\{  \rho \phi_i(\hat h,\hat \ell)  + i \frac{\mathcal E }{Sg (1- \chi )(\hat h - h_0)}   \frac{\partial \phi_i}{\partial h} (\hat h,\hat \ell) - \beta^{max} \, \sqrt{2g( \hat h - h_0)^+ }
 \left(  \frac{\partial \phi_i}{\partial h}(\hat h,\hat \ell) \right)^- \right. \\
& \quad\quad\quad
 - a(b- \hat \ell)\, \frac{\partial \phi_i}{\partial \ell}(\hat h,\hat \ell)  -  \hat \ell \int_{\Pi} \Big[ \phi_i(\hat h+z,\hat \ell+cz) - \phi_i(\hat h,\hat \ell) \Big] \pi(dz)  \\
& \quad\quad\quad   +    \left( \mu -  i \frac{\mathcal E }{S g(1- \chi)} \;\; \frac{1}{\hat h-h_0 }  \right)^+ \;\left(  \frac{\partial \phi_i }{\partial h} (\hat h,\hat \ell) \right)^+ \;  \mathbbm{1}_{\hat \ell \leq \underline{\ell} } \\
 & \quad\quad\quad + f_+ \left((\hat h- h_{+})^{+}\right) + f_- \left((h_- - \hat h)^{+}\right) - i \mathcal E \; ; \; \phi_i (\hat h,\hat \ell) - \phi_{1-i}(\hat h,\hat \ell)  + \kappa \Big\} \geq 0  \ .
 \end{split}
\end{equation}
By taking the convention of an immediate switching control $\tau_0=0$, $\tau_n=\infty$ for $n\geq 1$ and $\theta=0$ in \eqref{DPPeq} we get $v_i(h,\ell)\geq v_{1-i}(h,\ell)-k$ and in particular $v_i(\hat h,\hat \ell) -  v_{1-i}(\hat h,\hat \ell)  + \kappa\geq 0$, then $\phi_i (\hat h,\hat \ell) - \phi_{1-i}(\hat h,\hat \ell)  + \kappa\geq 0$.\\
We have then to show the other inequality. Given an initial regime value $i\in \{0,1\}$ consider the initial state-regime $(\hat h, \hat \ell,i)$ and take an arbitrary control $\alpha=\left( (\beta_t)_{t\geq 0}, \{ \tau_n \}_{n\in \N} \right)$. 
For all $\epsilon >0$, let $B_\epsilon(\hat h, \hat \ell)$ the ball  centred in $(\hat h, \hat \ell)$ of radius $\epsilon$ according with $L^\infty$-distance,
  and we set  $\theta_\epsilon :=\inf \left\{t\geq 0, \left(H^{(\hat h,\hat \ell;i)}_t, \lambda^{(\hat h,\hat \ell;i)}_t\right) \notin B_\epsilon(\hat h, \hat \ell)\right\}$, i.e. the exit time of the pair $\left(H^{(\hat h,\hat \ell;i)}_t, \lambda^{(\hat h,\hat \ell;i)}_t\right)$ from the ball $B_\epsilon(\hat h, \hat \ell)$. 
For the sake of simplicity reason we will denote $\left(H^{(\hat h,\hat \ell;i)}_t, \lambda^{(\hat h,\hat \ell;i)}_t\right)$ simply by $\left(H_t, \lambda_t\right)$. 
Without loss of generality $B_\epsilon(\hat h, \hat \ell)\subset \mathring D$ and then $\theta_\epsilon  < T$. 
We apply the dynamic programming principle \eqref{DPPeq} up to $\theta_\epsilon\wedge \tau_1\wedge s$ where $s>0$ and we obtain 
\begin{equation}		\label{DPPeq1}
\begin{split}
 \phi_i(\hat h,\hat \ell)& = v_i(\hat h,\hat \ell)\\
 &\geq \mathbb{E}^{i,\hat h, \hat \ell}\Bigg[ \int_0^{(\theta_\epsilon\wedge \tau_1\wedge s)} e^{- \rho t} \mathcal E \, I_t  dt
- \int_0^{(\theta_\epsilon\wedge \tau_1\wedge s)^-} e^{- \rho t} \kappa \; dN_t  \\
& \hspace{2cm} -  \int_0^{\theta_\epsilon\wedge \tau_1\wedge s} e^{-  \rho t}  \; \Bigg( f_{+}\left(\left(H_t- h_{+}\right)^{+}\right)  + f_{-}\left(\left(h_{-} - H_t\right)^{+}\right) \Bigg) dt \\
&\hspace{2cm} + e^{-\rho (\theta_\epsilon\wedge \tau_1\wedge s)} v_{I_{\theta_\epsilon\wedge \tau_1\wedge s}} \left(H_{\theta_\epsilon\wedge \tau_1\wedge s},\lambda_{\theta_\epsilon\wedge \tau_1\wedge s} \right)\Bigg]\\
& \geq \mathbb{E}^{i,\hat h, \hat \ell}\Bigg[ \int_0^{(\theta_\epsilon\wedge \tau_1\wedge s)} e^{- \rho t} \mathcal E \, i \; dt + e^{-\rho (\theta_\epsilon\wedge \tau_1\wedge s)} v_{i} \left(H_{\theta_\epsilon\wedge \tau_1\wedge s},\lambda_{\theta_\epsilon\wedge \tau_1\wedge s} \right)
\\
& \hspace{2cm} -  \int_0^{(\theta_\epsilon\wedge \tau_1\wedge s)} e^{-  \rho t}  \; \Bigg( f_{+}\left(\left(H_t- h_{+}\right)^{+}\right)  + f_{-}\left(\left(h_{-} - H_t\right)^{+}\right) \Bigg) dt  \Bigg]\\\end{split}
\end{equation}
where we have used the definition if $\tau_1$ and the fact that $\theta_\epsilon  < T$. A direct application of the It\^o formula in \cite[II Theorems 31,33]{Protter} to 
$e^{-\rho t} \phi_{i}\left(H_t, \lambda_t\right)$ for c\`ad-l\`ag semimartingales between $0$ and $\Theta:= \theta_\epsilon\wedge \tau_1\wedge s$ gives
\[
\begin{split}
e^{-\rho \Theta }\phi_{i}\left(H_{\Theta^-}, \lambda_{\Theta^-} \right)&= \phi_{i}\left(\hat h, \hat \ell \right)+\int_{0}^{\Theta} e^{-\rho t} \left( \mathcal{L}^{(i,\beta )} \phi_{i}-\rho \phi_{i}\right)\left( H_{t}, \lambda_{t} \right) dt\\
&-\int_{0}^{\Theta} e^{-\rho t} \lambda_{t^-} \int_{\Pi} \left[ \phi_i\left(H_{t^-}+z,\lambda_{t^-}+cz\right) - \phi_i \left( H_{t^-},\lambda_{t^-}\right)\right] \pi(dz) \, dt\\
&+\sum_{0< t<\Theta} e^{-\rho t} 
\left\{\phi_i\left(H_t,\lambda_t \right) -\phi_i\left(H_{t^-}, \lambda_{t^-} \right) 
\right\}.
\end{split}
\]
By considering the expectation of previous equality we get 
\begin{equation}   \label{DPPeq2}
\begin{split}
&\mathbb{E}^{i,\hat h, \hat \ell}\Bigg[ e^{-\rho \Theta }\phi_{i}\left(H_{\Theta^-}, \lambda_{\Theta^-} \right)\Bigg]
=\phi_{i}\left(\hat h, \hat \ell \right)+ \mathbb{E}^{i,\hat h, \hat \ell}\Bigg[ \int_{0}^{\Theta} e^{-\rho t} \left( \mathcal{L}^{(i,\beta )} \phi_{i}-\rho \phi_{i}\right)\left( H_{t^-}, \lambda_{t^-} \right) dt\Bigg] .\\
\end{split}
\end{equation}
Combing \eqref{DPPeq1} and \eqref{DPPeq2} we obtain
\begin{equation}   \label{DPPeq3}
\begin{split}
&\mathbb{E}^{i,\hat h, \hat \ell}\Bigg[ \int_{0}^{\Theta} e^{-\rho t} \left\{ \rho \phi_{i}\left( H_{t}, \lambda_{t} \right)-\mathcal{L}^{(i,\beta )} \phi_{i}\left( H_{t}, \lambda_{t} \right)  
-\mathcal E i + \left[ f_{+}\left(\left(H_t- h_{+}\right)^{+}\right)  + f_{-}\left(\left(h_{-} - H_t\right)^{+}\right)  \right]
 \right\} dt\Bigg]\\
 & \geq  \mathbb{E}^{i,\hat h, \hat \ell}\Bigg[ e^{-\rho \Theta} 
 \left\{v_{i} \left(H_{\Theta},\lambda_{\Theta} \right)-\phi_{i}\left(H_{\Theta^-}, \lambda_{\Theta^-} \right)   \right\}\Bigg]\geq 0
%
\end{split}
\end{equation}
From the definition of $\theta_\epsilon $, the integrand in the first line of \eqref{DPPeq3} is bounded. So dividing first term in previous inequality by $s$ and taking $s$ to $0$, we may apply the dominated convergence theorem obtaining 
\begin{equation}
 \rho \phi_{i}\left( \hat h, \hat \ell \right)-\mathcal{L}^{(i,\beta )} \phi_{i}\left( \hat h, \hat \ell \right)  
-\mathcal E i + \left( f_{+}\left(\left(\hat h- h_{+}\right)^{+}\right)  + f_{-}\left(\left(h_{-} - \hat h\right)^{+}\right)  \right)\geq 0
\end{equation}
for the arbitrary control $\alpha$ fixed at the beginning. By taking the supremum over all the admissible controls, we get the second supersolution inequality.

\emph{Proof of subsolution property. } 
We prove the subsolution property, i.e. fix $i\in \{1,0\}$ and let $(\hat h, \hat \ell)\in \mathring D$ and $\phi_i\in C^1(D,\R)$ such that $(\hat h, \hat \ell)$ is a maximum of $v_i-\phi_i$ with  $\phi_i(\hat h, \hat \ell)=v_i(\hat h, \hat \ell)$. Of course $(v_i-\phi_i)(h,\ell)\leq 0$ on a neighbourhood of 
$(\hat h, \hat \ell)$. 
We have to prove that 
\begin{equation}
\begin{split}
& \min \left\{  \rho \phi_i(\hat h,\hat \ell)  + i \frac{\mathcal E }{Sg (1- \chi )(\hat h - h_0)}   \frac{\partial \phi_i}{\partial h} (\hat h,\hat \ell) - \beta^{max} \, \sqrt{2g( \hat h - h_0)^+ }
 \left(  \frac{\partial \phi_i}{\partial h}(\hat h,\hat \ell) \right)^-   \right.  \\
& \quad\quad\quad
 - a(b- \hat \ell)\, \frac{\partial \phi_i}{\partial \ell}(\hat h,\hat \ell)  -  \hat \ell \int_{\Pi} \Big[ \phi_i(\hat h+z,\hat \ell+cz) - \phi_i(\hat h,\hat \ell) \Big] \pi(dz)  \\
& \quad\quad\quad  +    \left( \mu -  i \frac{\mathcal E }{S g(1- \chi)} \;\; \frac{1}{\hat h-h_0 } \right)^+\left(  \frac{\partial \phi_i }{\partial h} (\hat h,\hat \ell) \right)^+ \; \mathbbm{1}_{\hat \ell \leq \underline{\ell} }
\\
 & \quad\quad\quad + f_+ \left((\hat h- h_{+})^{+}\right) + f_- \left((h_- - \hat h)^{+}\right) - i \mathcal E \; ; \; \phi_i (\hat h,\hat \ell) - \phi_{1-i}(\hat h,\hat \ell)  + \kappa \Big\} \leq 0  
 \end{split}
\end{equation}

We prove by contradiction. Suppose that the claim is not true. Then there exist an arbitrary control $\alpha=\left( (\beta_t)_{t\geq 0}, \{ \tau_n \}_{n\in \N} \right)$ and a neighbourhood of $(\hat h, \hat \ell)$ denoted by $B_\epsilon(\hat h, \hat \ell)$ and $\eta >0$ such that the following inequalities hold.
\begin{equation*}
\left\{ 
\begin{array}{rcl}
\displaystyle  \rho \phi_i(\hat h,\hat \ell)  + i \frac{\mathcal E }{Sg (1- \chi )(\hat h - h_0)}   \frac{\partial \phi_i}{\partial h} (\hat h,\hat \ell) - \beta^{max} \, \sqrt{2g( \hat h - h_0)^+ }
 \left(  \frac{\partial \phi_i}{\partial h}(\hat h,\hat \ell) \right)^-  \\
 \displaystyle  \quad\quad\quad
 - a(b- \hat \ell)\, \frac{\partial \phi_i}{\partial \ell}(\hat h,\hat \ell)  -  \hat \ell \int_{\Pi} \Big[ \phi_i(\hat h+z,\hat \ell+cz) - \phi_i(\hat h,\hat \ell) \Big] \pi(dz)  \\
 \displaystyle  \quad\quad\quad  +    \left( \mu -  i \frac{\mathcal E }{S g(1- \chi)} \;\; \frac{1}{\hat h-h_0 }  \right)^+ \left(  \frac{\partial \phi_i }{\partial h} (\hat h,\hat \ell) \right)^+ \;  \mathbbm{1}_{\hat \ell \leq \underline{\ell} }  \\
  \displaystyle \quad\quad\quad + f_+ \left((\hat h- h_{+})^{+}\right) + f_- \left((h_- - \hat h)^{+}\right) - i \mathcal E \;  &>& \eta
 \\
 \displaystyle  \phi_i (\hat h,\hat \ell) - \phi_{1-i}(\hat h,\hat \ell)  + \kappa  &>& \eta
 \end{array}
 \right.
\end{equation*}
Then by regularity of $\phi$ for all $(h,\ell)\in B_\epsilon (\hat h, \hat \ell)$, eventually smaller, we have
\begin{equation}\label{hypAss}
\left\{ 
\begin{array}{rcl}
\displaystyle  \rho \phi_i( h, \ell)  + i \frac{\mathcal E }{Sg (1- \chi )( h - h_0)}   \frac{\partial \phi_i}{\partial h} ( h, \ell) - \beta^{max} \, \sqrt{2g( h - h_0)^+ }
 \left(  \frac{\partial \phi_i}{\partial h}( h, \ell) \right)^-  \\
 \displaystyle  \quad\quad\quad
 - a(b-  \ell)\, \frac{\partial \phi_i}{\partial \ell}( h,\ell)  -   \ell \int_{\Pi} \Big[ \phi_i( h+z, \ell+cz) - \phi_i( h, \ell) \Big] \pi(dz)  \\
 \displaystyle  \quad\quad\quad  +    \left( \mu -  i \frac{\mathcal E }{S g(1- \chi)} \;\; \frac{1}{ h-h_0 }  \right)^+ \left(  \frac{\partial \phi_i }{\partial h} ( h, \ell) \right)^+ \;  \mathbbm{1}_{ \ell \leq \underline{\ell} }  \\
  \displaystyle \quad\quad\quad + f_+ \left((h- h_{+})^{+}\right) + f_- \left((h_- -  h)^{+}\right) - i \mathcal E \;  &>& \eta
 \\
 \displaystyle  \phi_i ( h,\ell) - \phi_{1-i}( h, \ell)  + \kappa  &>& \eta
 \end{array}
 \right.
\end{equation}

We consider the exit time $\theta_\epsilon :=\inf \left\{t\geq 0, \left(H_t, \lambda_t\right) \notin B_\epsilon(\hat h, \hat \ell)\right\}$.
Applying It\^o formula 
to 
$e^{-\rho t} \phi_{i}\left(H_t, \lambda_t\right)$ for c\`ad-l\`ag semimartingales between $0$ and $\gamma_\epsilon=\theta_\epsilon  \wedge \tau^-_1$ and taking the expectation, we obtain
\begin{equation}   \label{DPPeq4}
\begin{split}
&\mathbb{E}^{i,\hat h, \hat \ell}\Bigg[ e^{-\rho( \gamma_\epsilon)}\phi_{i}\left(H_{\gamma_\epsilon^-}, \lambda_{\gamma_\epsilon^-} \right)\Bigg]
=\phi_{i}\left(\hat h, \hat \ell \right)+ \mathbb{E}^{i,\hat h, \hat \ell}\Bigg[ \int_{0}^{\gamma_\epsilon} e^{-\rho t} \left( \mathcal{L}^{(i,\beta )} \phi_{i}-\rho \phi_{i}\right)\left( H_{t^-}, \lambda_{t^-} \right) dt\Bigg] .\\
\end{split}
\end{equation}
Combing hypothesis \eqref{hypAss} and \eqref{DPPeq4} we obtain
\begin{equation}  
\begin{split}
&\mathbb{E}^{i,\hat h, \hat \ell}\Bigg[ e^{-\rho \gamma_\epsilon}\phi_{i}\left(H_{\gamma_\epsilon^-}, \lambda_{\gamma_\epsilon^-} \right)\Bigg]
\leq  \phi_{i}\left(\hat h, \hat \ell \right)+ \mathbb{E}^{i,\hat h, \hat \ell}\Bigg[ \int_{0}^{\gamma_\epsilon} e^{-\rho t} \left( -\eta -\mathcal E i + \left( f_{+}\left(\left(H_t- h_{+}\right)^{+}\right)  + f_{-}\left(\left(h_{-} - H_t\right)^{+}\right)  \right) \right) dt\Bigg] .
\end{split}
\end{equation}
Then 
\begin{equation}  
\begin{split}
\phi_{i}\left(\hat h,  \ell \right)  \geq & \mathbb{E}^{i,\hat h, \hat \ell}\Bigg[ e^{-\rho \gamma_\epsilon}\phi_{i}\left(H_{\gamma_\epsilon^-}, \lambda_{\gamma_\epsilon^-} \right)\Bigg]\\
& +\mathbb{E}^{i,\hat h, \hat \ell}\Bigg[ \int_{0}^{\gamma_\epsilon} e^{-\rho t} \left( \eta +\mathcal E i - \left( f_{+}\left(\left(H_t- h_{+}\right)^{+}\right)  - f_{-}\left(\left(h_{-} - H_t\right)^{+}\right)  \right) \right) dt\Bigg] \\
\geq &\eta \ \mathbb{E}^{i,\hat h, \hat \ell}\Bigg[  \frac{1-e^{-\rho \gamma_\epsilon}}{\rho}\Bigg] +\mathbb{E}^{i,\hat h, \hat \ell}\Bigg[ e^{-\rho\theta_\epsilon }\phi_{i}\left(H_{\theta_\epsilon ^-}, \lambda_{\theta_\epsilon ^-} \right) \1_{\theta_\epsilon \leq  \tau^-_1}\Bigg]\\
& +\mathbb{E}^{i,\hat h, \hat \ell}\Bigg[ e^{-\rho\tau_1}\phi_{i}\left(H_{\tau_1^-}, \lambda_{\tau_1^-} \right) \1_{\tau_1\leq \theta_\epsilon  }\Bigg]
 \\
& +\mathbb{E}^{i,\hat h, \hat \ell}\Bigg[ \int_{0}^{\gamma_\epsilon} e^{-\rho t} \left( \mathcal E i - \left( f_{+}\left(\left(H_t- h_{+}\right)^{+}\right)  - f_{-}\left(\left(h_{-} - H_t\right)^{+}\right)  \right) \right) dt\Bigg] .
\end{split}
\end{equation}
By the fact that $\phi_i\geq v_i$ on $B_\epsilon$ we obtain
\begin{equation}  
\begin{split}
\phi_{i}\left(\hat h, \hat \ell \right)  \geq &  \eta \ \mathbb{E}^{i,\hat h, \hat \ell}\Bigg[  \frac{1-e^{-\rho \gamma_\epsilon}}{\rho}\Bigg] 
  +\mathbb{E}^{i,\hat h, \hat \ell}\Bigg[ e^{-\rho\tau_1}\left( v_{1-i}\left(H_{\tau_1^-}, \lambda_{\tau_1^-} \right)+\eta -k \right) \1_{\tau_1\leq \theta_\epsilon }\Bigg]\\
&  
+  \mathbb{E}^{i,\hat h, \hat \ell}\Bigg[ e^{-\rho\theta_\epsilon }v_{i}\left(H_{\theta_\epsilon ^-}, \lambda_{\theta_\epsilon ^-} \right) \1_{\theta_\epsilon \leq  \tau^-_1}\Bigg]
 \\
& +\mathbb{E}^{i,\hat h, \hat \ell}\Bigg[ \int_{0}^{\gamma_\epsilon} e^{-\rho t} \left( \mathcal E i - \left( f_{+}\left(\left(H_t- h_{+}\right)^{+}\right)  - f_{-}\left(\left(h_{-} - H_t\right)^{+}\right)  \right) \right) dt\Bigg] .
\end{split}
\end{equation}
Identifying the different cases, taking the supremum over all admissible controls $\alpha$ and using the Dynamic Programming Principle \eqref{DPPeq} we have the following inequality.
\[
\phi_{i}\left(\hat h, \hat \ell \right) \geq v_i (\hat h, \hat \ell)+ \eta \; \mathbb{E}^{i,\hat h, \hat \ell}\Bigg[  \frac{1-e^{-\rho \gamma_\epsilon}}{\rho} + e^{-\rho\tau_1}  \1_{\tau_1\leq \theta_\epsilon }  \Bigg].
\]
We now focus on the random time $\theta_\epsilon$. According with the definition of $B_\epsilon$
we have that  $\theta_\epsilon =  \theta^{H+}_\epsilon \wedge \theta^{H-}_\epsilon \wedge \theta^{\lambda+}_\epsilon \wedge \theta^{\lambda-}_\epsilon $, where 
 $\theta^{H+}_\epsilon = \inf\{t \, $ such that $  H_t \geq \hat h + \epsilon \}$, 
 $\theta^{H-}_\epsilon = \inf\{t \, $ such that $  H_t \leq \hat h - \epsilon \}$,
 $\theta^{\lambda+}_\epsilon = \inf\{t \, $ such that $  \lambda_t \geq  \hat \ell + \epsilon \}$ and 
$\theta^{\lambda-}_\epsilon = \inf\{t \, $ such that $  \lambda_t \leq \hat \ell - \epsilon \}$.
For $\epsilon< c/2\min\{z_i \}$, $ \theta^{\lambda+}_\epsilon = \Theta_1$. Moreover $\theta^{H+}_\epsilon \in \{\Theta_i\}_{i\in \mathbb{N}}$. Restricted on $[0, \Theta_1)$, the stochastic processes $H$ and $\lambda$ are deterministic non-increasing continuous function.
Then, restricted to the event $\theta^{H-}_\epsilon \leq \Theta_1$, the random time $\theta^{H-}_\epsilon$ is a deterministic increasing (extended)-function of $\epsilon$ taking values on $(0, \infty]$. A similar argument works for 
$\theta^{\lambda-}_\epsilon$ with the main difference that the 
deterministic increasing function is finite. We will call $g(\epsilon)$ the minimum of these two deterministic functions, that is   $g(\epsilon)$ is the deterministic hitting of $\theta^{H-}_\epsilon \wedge \theta^{\lambda-}_\epsilon$ restricted on $[0, \Theta_1)$.
Focusing on $\Theta_1$, we remark that the first jump of a Hawkes process coincides with the first jump of a time-inhomogeneous Poisson process. Moreover, the intensity of the Hawkes process is not-increasing up to the first jump then we have $\mathbb{P}^{\hat \ell}[\Theta_1>\delta] \geq e^{-\hat \ell \, \delta }$. We now fix $\delta = g(\epsilon)$ and we have
\begin{eqnarray*}
\mathbb{E}^{i,\hat h, \hat \ell}\Bigg[  \frac{1-e^{-\rho \gamma_\epsilon}}{\rho} + e^{-\rho\tau_1}  \1_{\tau_1\leq \theta_\epsilon }  \Bigg] &>&  
\mathbb{E}^{i,\hat h, \hat \ell}\Bigg[  \left\{ \frac{1-e^{-\rho \gamma_\epsilon}}{\rho} + e^{-\rho\tau_1}  \1_{\tau_1\leq \theta_\epsilon } \right\} \1_{\Theta_1>g(\epsilon)}  \Bigg] \\
 &>&  
\mathbb{E}^{i,\hat h, \hat \ell}\Bigg[ \frac{1-e^{-\rho \gamma_\epsilon}}{\rho} \1_{\Theta_1 \wedge \tau_1>g(\epsilon) }  + e^{-\rho\tau_1}  \1_{\tau_1\leq g(\epsilon) \leq \Theta_1 }  \Bigg] \\
&>&  
\frac{1-e^{-\rho  g(\epsilon)}}{\rho} e^{-\hat \ell \, g(\epsilon)} \mathbb{P}^{i,\hat h, \hat \ell}\left[ \tau_1>g(\epsilon) \right]  + e^{-(\rho+\hat \ell)  g(\epsilon)}   \mathbb{P}^{i,\hat h, \hat \ell} \left[ \tau_1\leq g(\epsilon)   \right] 
\end{eqnarray*}
Since $g(\epsilon)>0$ we obtain that there exists $c_0>0$ such that
$ \phi_{i}\left(\hat h, \hat \ell \right) \geq v_i (\hat h, \hat \ell)+ \eta \; c_0 $, 
that is a contradiction with the initial hypothesis. 
\end{proof}

\section{Numerical results}\label{num_sec}		\label{numericalResu}

\setcounter{equation}{0} \setcounter{Assumption}{0} \setcounter{Theorem}{0} \setcounter{Proposition}{0} \setcounter{Corollary}{0}
\setcounter{Lemma}{0} \setcounter{Definition}{0} \setcounter{Remark}{0}
This section deals with approximations of solutions and numerical studies of our optimization problem for dams. The first subsection deals with the numerical scheme used to approximate the HJB equation while the second one discusses the numerical results. 

\subsection{Approximation of solutions}

To solve the HJB equation (\ref{HJB}) arising from the
stochastic control problem (\ref{eqn-value-function}), we choose to use a deterministic approach based on a finite difference scheme which leads
to the resolution of a controlled Markov chain problem. This technique was widely popularised by Kushner and Dupuis (2013).
The convergence of the solution of the numerical scheme towards the solution of the HJB equation, when the space step goes to zero, can be shown using standard arguments, i.e. it satisfies monotonicity, consistency and stability properties. Similar numerical schemes, involving a controlled Markov chain problem, are exploited in operational research, see for instance
 Cao, Li, and Yan (2012), Jin, Yin and Zhu (2012), Parpas et Webster (2014),  Cosso, Marazzina and Sgarra (2015) and Ga\"{\i}gi, Ly Vath and Scotti (2022).

We first localise the problem on a discretised grid. Let $j$ and $k$ be the discretisation steps along the directions $h$ and $\ell$ respectively. We define the space grid as $\Gc_{j,k}:=\{h_{min},h_{min}+j,h_{min}+2j,..,h^{max}\}\times\{\ell_{min},\ell_{min}+k,\ell_{min}+2k,...,\ell_{max}\}$, where $h_{min}$, $h^{max}$, $\ell_{min}$ and $\ell_{max}$ are non-negative constants. We choose a discrete probability distribution for $\pi(dx)$ arising in the HJB equation of a sample space $\{z_1,z_2,z_3\}$ and the associated probabilities $\{\pi_1,\pi_2,\pi_3\}$.\\
For sake of readability we introduce the following quantities and set:
\beqs
& & y_1:=(h+j,\ell),  \quad
y_2:=(h,\ell+k),  \quad
y_3:=(h-j,\ell),  \quad
y_4:=(h,\ell-k),\\
& & y_5:=(\min(h^{max},h+ z_1),\min(\ell_{max},\ell+c z_1)),  \quad
y_6:=(\min(h^{max},h+ z_2),\min(\ell_{max},\ell+c z_2)),\\
& & y_7:=(\min(h^{max},h+ z_3),\min(\ell_{max},\ell+c z_3)), \\
& & \mu^i_h:=\displaystyle  - i  \frac{\mathcal E }{S g(1- \chi)} \frac{1}{h - h_0 }
  - \beta \sqrt{2 g (h - h_0 )}, \quad \mu_{\ell}:= \displaystyle a(b-\ell), \;\\
  & & B(i,h,\ell):=\{\beta_{i,h,\ell}; \beta^{max}\}, \quad \beta_{i,h,\ell}:= \max \left( \mu -  \varphi(h)i ; \, 0 \right)\mathbbm{1}_{\ell \leq \underline{\ell} }
\enqs
 For $(h,\ell)$ in the space grid $\Gc_{j,k}$ we consider approximations of the following form:
\beqs
 \frac{\partial v_i}{\partial h}(h,\ell)&\approx&  \frac{v_i (h+ j,\ell) - v_i (h,\ell)}{j}\mathds{1}_{\mu_h\geq0}-\frac{v_i (h-j,\ell) - v_i (h,\ell)}{j}\mathds{1}_{\mu_h<0},\\
 \frac{\partial v_i }{\partial \ell}(h,\ell)&\approx&   \frac{v_i  (h,\ell+ k) - v_i (h,\ell)}{k}\mathds{1}_{\mu_\ell\geq0}-\frac{v_i  (h,\ell- k) - v_i (h,\ell)}{k}\mathds{1}_{\mu_\ell<0}.
 \enqs

 Thus, using the above notations and applying a finite difference scheme, the HJB equation (\ref{HJB}) can be formulated as the following:
\beq\label{DiscHJB}
v_i(h,\ell)=\max\Bigg\{\max_{\beta\in B(i,h,\ell)}\frac{\sum_{m=1}^{7}p_mv(y_m)+G^{i}_h\Delta t^{h,k}}{1+\rho\Delta t^{h,k}},v_{1-i}(h,\ell)-\kappa\Bigg\},
\enq
where
$$
\begin{array}{rclrcl}
p_1(h,\ell)\!\!\!&:=&\!\!\!\displaystyle \frac{k(\mu_{h}^{i})^{+}}{Q^{j,k}}, \quad &
p_2(h,\ell)\!\!\!&:=&\!\!\!\displaystyle\frac{j\mu_{\ell}^{+}}{Q^{j,k}},\\
p_3(h,\ell)\!\!\!&:=&\!\!\!\displaystyle\frac{k(\mu_{h}^{i})^{-}}{Q^{j,k}}, \quad &
p_4(h,\ell)\!\!\!&:=&\!\!\!\displaystyle\frac{j\mu_{\ell}^{-}}{Q^{j,k}},\\
p_5(h,\ell)\!\!\!&:=&\!\!\!\displaystyle\frac{jk\ell\pi_1}{Q^{j,k}}, \quad &
p_6(h,\ell)\!\!\!&:=&\!\!\!\displaystyle\frac{jk\ell\pi_2}{Q^{j,k}},\\
p_7(h,\ell)\!\!\!&:=&\!\!\!\displaystyle\frac{jk\ell\pi_3}{Q^{j,k}}, \quad & \Delta t^{j,k}(h,\ell)\!\!\!&:=&\!\!\!\displaystyle \frac{jk}{Q^{j,k}},\\
\end{array}
$$
$$
Q^{j,k}(h,\ell):=\displaystyle (|\mu^i_h| k+|\mu_\ell| j)+ jk\ell.
$$
$$ G^{i}_h := i \mathcal{E}- f_{+}\left(\left(h- h_{+}\right)^{+}\right)  - f_{-}\left(\left(h_{-} - h\right)^{+}\right) $$
with the notation $(.)^+$ (resp. $(.)^-$) representing the positive (resp. negative) part  of a given function.\\
To compute explicitly the approximated solution of the discrete problem (\ref{DiscHJB}) we use the following iterative scheme:
\beq\label{IterHJB}
v_i^{(n+1)}(h,\ell)&=&\max\Bigg\{\max_{\beta\in B(i,h,\ell)}\frac{\sum_{m=1}^{7}p_mv_i^{(n)}(y_m)+G^{i}_h\Delta t^{h,k}}{1+\rho\Delta t^{h,k}},v_{1-i}^{(n)}(h,\ell)-\kappa\Bigg\},\\
v_i^{0}(h,\ell)&=&0,
\enq
recalling the Dirichlet boundary condition $v_n(h^{max},\ell)=0$ for all $n$ and Neumann condition on the other bounds.
The above iterative scheme is explicit and fully implementable on the enlarged grid $\Gc_{h,k}^{+}:=\{-j,0,..,h^{max}\}\times\{\ell_{min}-k,..,\ell_{max}+k\}$.\\
Using the following parameters, about 32 seconds are necessary to obtain the approximated value function and policy using Intel$^{\tiny{\hbox{TM}}}$Core i7 CPU at 2.70 Ghz  with 8 GB of RAM.

\subsection{Numerical Results}\label{numerical-examples}

We focus on a dam, which height with respect to the river level is lower than 100 meters, that is the typical height, the tallest dam in the world is Jinping-I with 305 meter height. We focus essentially 
on mountain dams with a limited volume and a reduced inflow area, that is the usual situation in high mountain areas like the alps in Europe and the West coast in America. The mountain shape and the reduced inflow area contribute both to sharp fluctuation on raining therefore a discontinuous driver can better capture the evolution.

The electricity production essentially depends on the water volume collected on the artificial lake and then on the geography of the dam site. For sake of readability, we renormalize the surface $S$ to the unit, that is the value function is expressed for unit of square meters of the basin surface.
Both the penalization functions $f_-$ and $f_*$ (the one for the dangerous high water level and the one for the touristic scope) are assumed quadratic\footnote{Both penalization functions are Lipschitz since functions on a finite interval.} with a coefficient $1/2 \times 10^{-3}$.

The rest of the parameters are resumed in Table \ref{parameter-table} where they are split according to their kind, i.e. dam construction, external/political penalisation thresholds, hydro production setup, raining evolution, gravity constant, discount factor and discretisation mesh setup. Values are chosen in a synthetic way such that we capture every aspect of the dynamics.

\begin{table}[h!]	
  \centering
    \begin{tabular}{|c|c|c|}
    \hline
    parameter & meaning & value \\
    \hline
    \hline
            $h^{max}$ & maximal water level inside the dam & $100$ \\
         \hline
            $h^{min}$ & bottom level of the dam& $0$ \\
            \hline
                     $h_0$ & turbine position with respect to the dam bottom  & -1\\ 
                     \hline
        ${\beta^{max}}$ & maximum flow of the spillover & $1.2$\\
\hline
    \hline
         $ h_+$  & dangerous water level & $80$ \\
        \hline 
         $h_-$  & touristic minimal quote & $50$ \\
            \hline
         $ \underline{\ell}$  & threshold defining low-flow period & $1$ \\      
          \hline
        $\mu$  & minimum outflow during low-flow period  & $0.4$\\  
        \hline
        $P$  & Penalisation for dam failure  & $0$\\  
        \hline
        \hline
        $\mathcal E$ &  Energy production normalized for unity of surface   & $3$\\
        \hline
         $S$ & renormalised surface & $1$\\ 
        \hline
        $1- \chi$ & standard efficiency for a turbine of type \emph{Francis}& $0.95$\\
         \hline
         $\kappa$ & switching cost & $3$\\
          \hline
         \hline
         $a$ & intensity exponential decay between two rains & $0.3$ \\
         \hline
         $b$ &  minimal intensity raining & $0.01$\\
         \hline
         $c$ & self exciting effect & $0.1$ \\
         \hline
         \hline
           $g$ & gravity acceleration & $9.806$\\
         \hline
              \hline
        $\rho$ & discount rate for the value function & $0.2$\\
                     \hline \hline
         $nh$ & number of points for $H$ mesh & $100$\\
    \hline
    $nl$ & number of points for $\lambda$ mesh & $100$\\
    \hline
    $l_{max}$ & maximal intensity & $3$ \\
      \hline
         $ l_{min}$  & minimal intensity & $0.01$ \\
       \hline
    \end{tabular} 	
    \caption{Central values for the parameters used for the numerical tests.}\label{parameter-table}
   \end{table}
   
For simplicity, the law of raining quantities is assumed discrete taking three values, see Table \ref{table-2}. That is after a storm the level of the water inside the dam increases by  $\{10, \, 15, \, 20\}$ meters and the average inflow due to a storm is 13 meters.

\begin{table}[h!]
  \centering
    \begin{tabular}{|c|c|c|}
    \hline
    parameter & value of $J$ & probability \\
    \hline
    \hline
         $z_1$ &  10 & $p_1=0.5$ \\
         \hline
         $z_2$ &  15 & $p_2=0.4$ \\
         \hline
         $z_3$ &  20 & $p_3=0.1$ \\
         \hline 
         \hline
    \end{tabular}
    \caption{Probability values for the numerical tests.} \label{table-2}
   \end{table}

We begin the numerical analysis by computing the two value functions; see Figure \ref{Fig1}. These functions represent the numerically estimated value functions over the state space $(h,\ell)$. We observe that both value functions are increasing in both rainfall intensity and water level, except at very high water levels, where the risk of dam failure becomes significant. In this region, the value functions decrease with respect to intensity, since during storm periods a high water level in the reservoir substantially increases the probability of dam failure.
The costs associated with the touristic quota penalise the value function particularly during drought periods, highlighting the economic relevance of the proposed framework. In particular, the constraint on the minimum outflow during low-flow regimes plays a crucial role in determining the optimal policy.
Focusing on the main diagonal of the state space (from low intensity, low water level to high intensity, high water level), the value functions exhibit concavity, indicating that risks are concentrated at the two extremes. In contrast, along the opposite diagonal (from high intensity, low water level to low intensity, high water level), the value functions are relatively flat or slightly convex, emphasising the stabilising effect of water regulation provided by the dam.

\begin{center}
\begin{figure}[ht]
    \centering
     \includegraphics[scale=0.3]{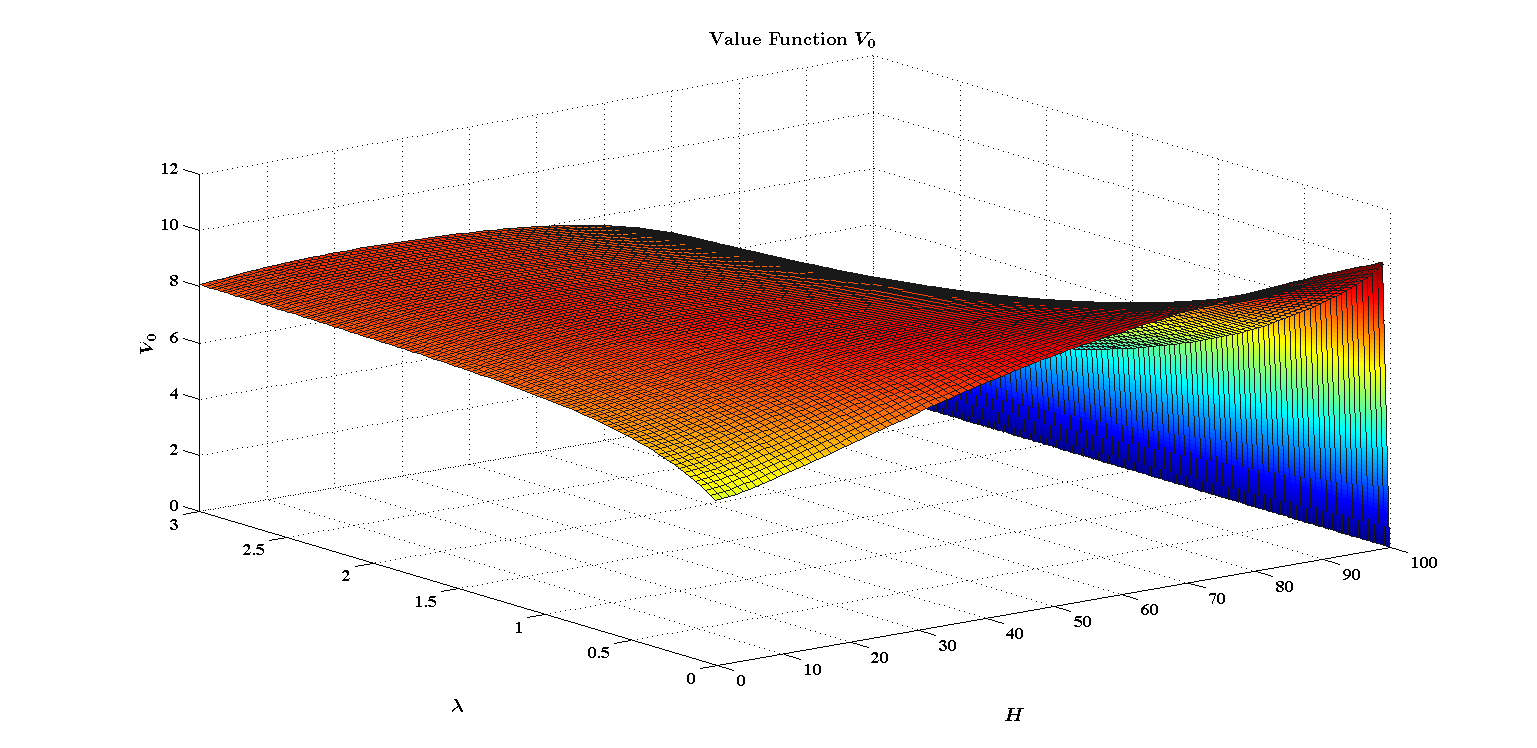} 
        \includegraphics[scale=0.3]{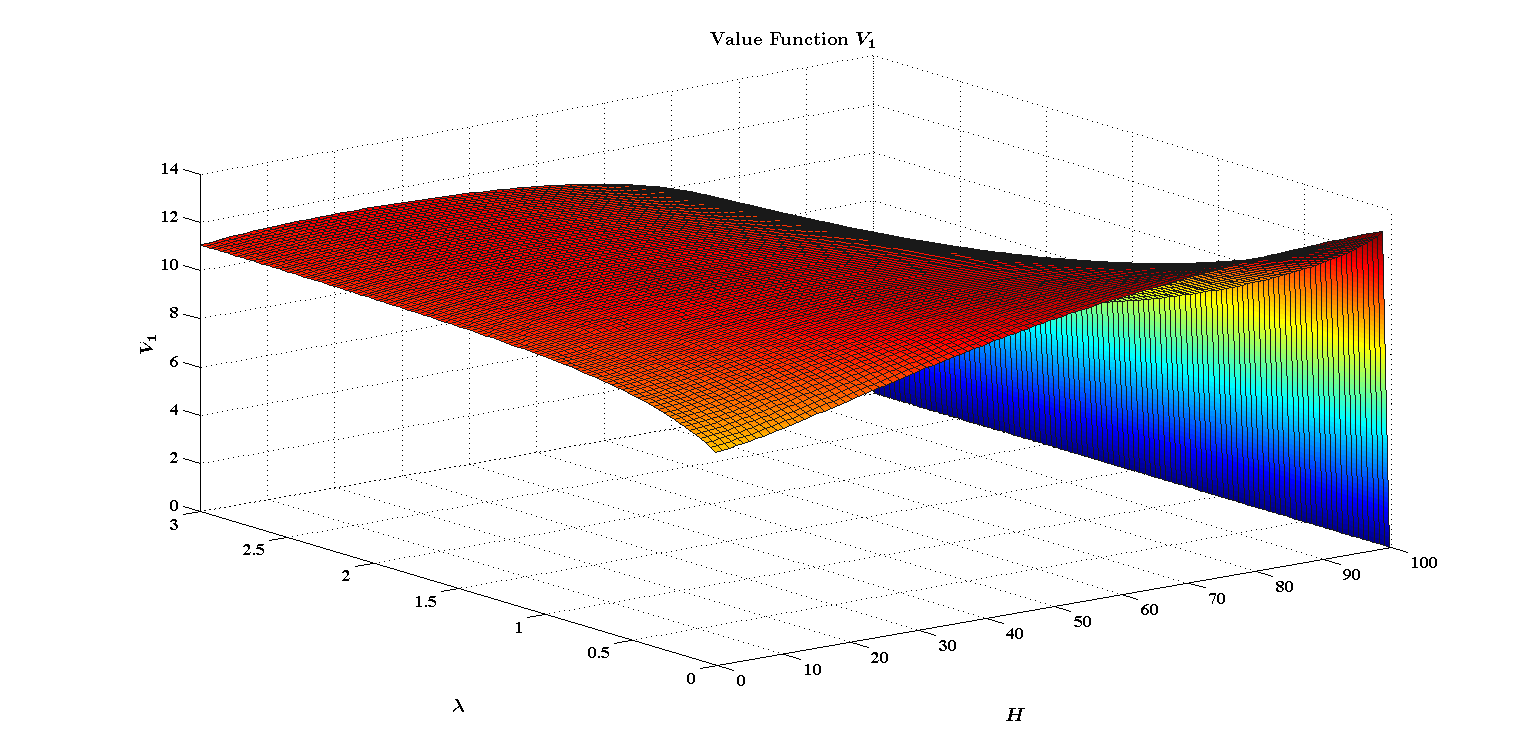} 
    \caption{\emph{Value functions $v_0$ and $v_1$.}}
    \label{Fig1}
      \end{figure}
\end{center}

Figure \ref{Fig2_R} depicts the continuation region together with the optimal switching policy between the on and off states of the turbine. For the current parameter configuration, most notably the relatively high switching costs, it is never optimal to stop electricity production. Conversely, when the water level is sufficiently high, it is optimal to initiate production and subsequently increase the outflow.
In the first panel, the initial regime is $i=0$, corresponding to the turbine being switched off. The red region indicates the switch action, i.e. a switch in the operational state of the turbine. The blue region corresponds to the no-switch region, meaning that no switch is undertaken and the turbine remains off.
In the second panel, where the turbine is initially in the on state, the blue region corresponds to the control action $i=0$, indicating that it is optimal to maintain the current production regime. Notice that the figure represents the optimal switching decision and not the current operating state of the turbine. A value equal to one (red area) corresponds to an immediate switch, whereas a value equal to zero (blue area) corresponds to continuation in the current regime.

\begin{center}
\begin{figure}[ht]
    \centering
     \includegraphics[scale=0.3]{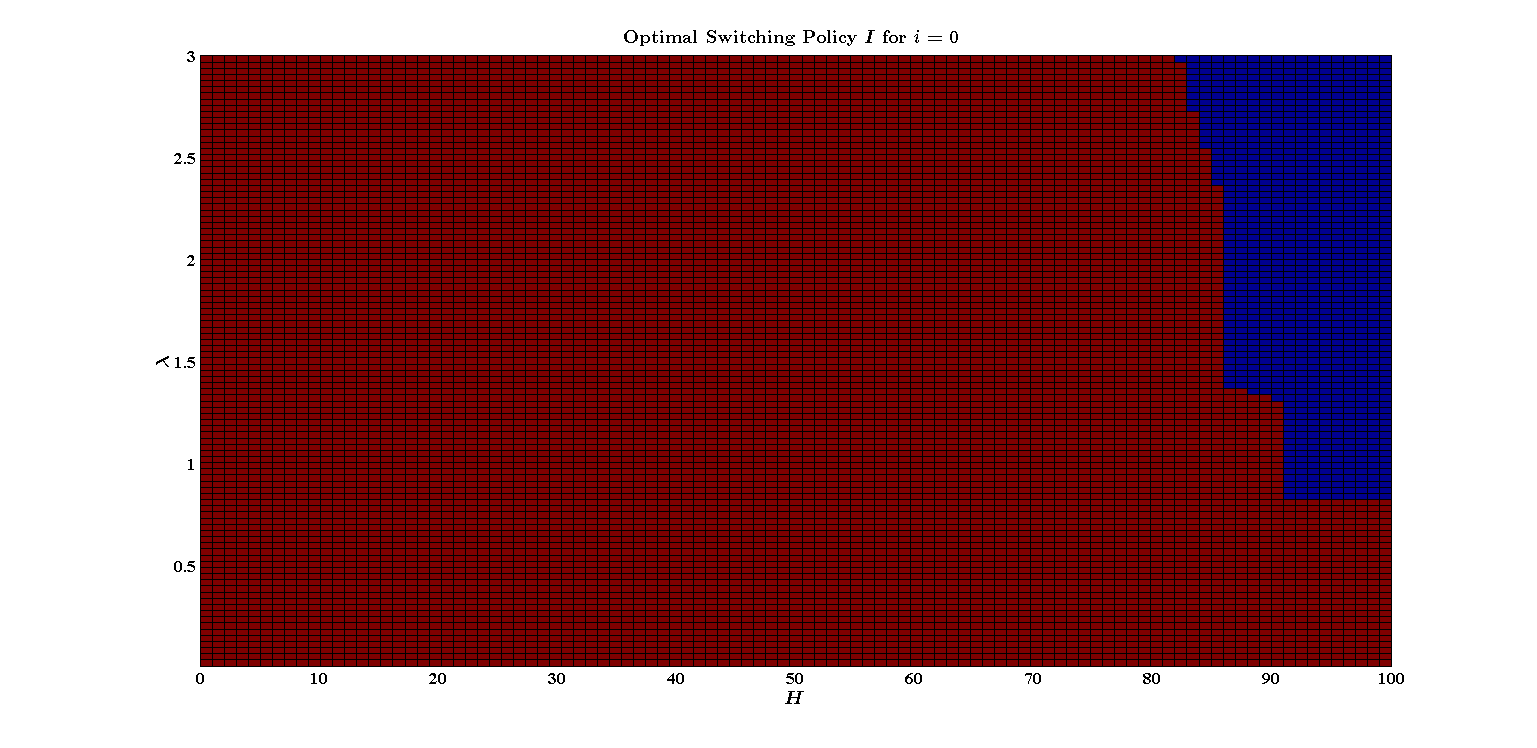} 
        \includegraphics[scale=0.3]{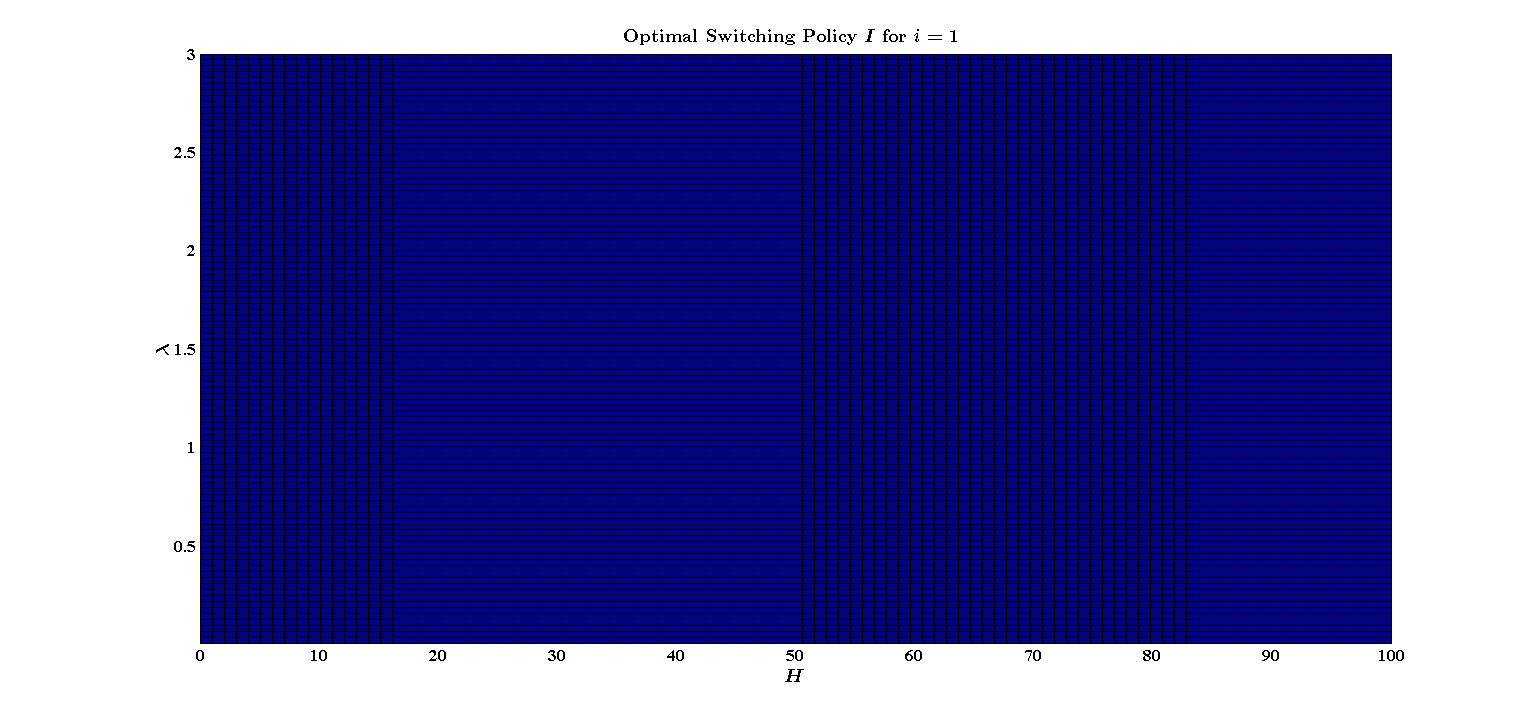} 
    \caption{\emph{Optimal Switching policy for $i=0$ and $i=1$.}}
    \label{Fig2_R}
      \end{figure}
\end{center}

This behaviour is confirmed by Figure \ref{Fig3_R}, which shows that the optimal outflow through the spillway is set at its maximum level for very high water levels in the reservoir. We note that only two colours appear for the a priori continuous control $\beta$. In Figure \ref{Fig3_R}, red corresponds to $\beta=\beta^{\max}$, while blue corresponds to $\beta=\beta^{\min}$.
This indicates that the optimal water-release policy is of bang?bang type. Moreover, the optimal threshold, defined as the water level above which it is optimal to release water at the maximum rate allowed by the spillway, decreases with rainfall intensity. This reflects the fact that during storm periods it is optimal to lower the reservoir level in order to mitigate the risk of dam failure due to storm clustering. In particular, it is optimal to keep the spillway open even when the water level is below the touristic quota. This result further confirms the central role of active water management in dam operations.
During low-flow regimes, that is, when the intensity is below $\underline{\ell}$, the policy-imposed minimum release $\mu$ is ensured through a combination of spillway discharge and turbine flow. Since the available potential energy decreases with the water level, electricity production requires a larger flow to maintain the target operating frequency when the reservoir level is low. Consequently, the flow released through the spillway decreases as the water level decreases.

\begin{center}
\begin{figure}[h]
    \centering
     \includegraphics[scale=0.38]{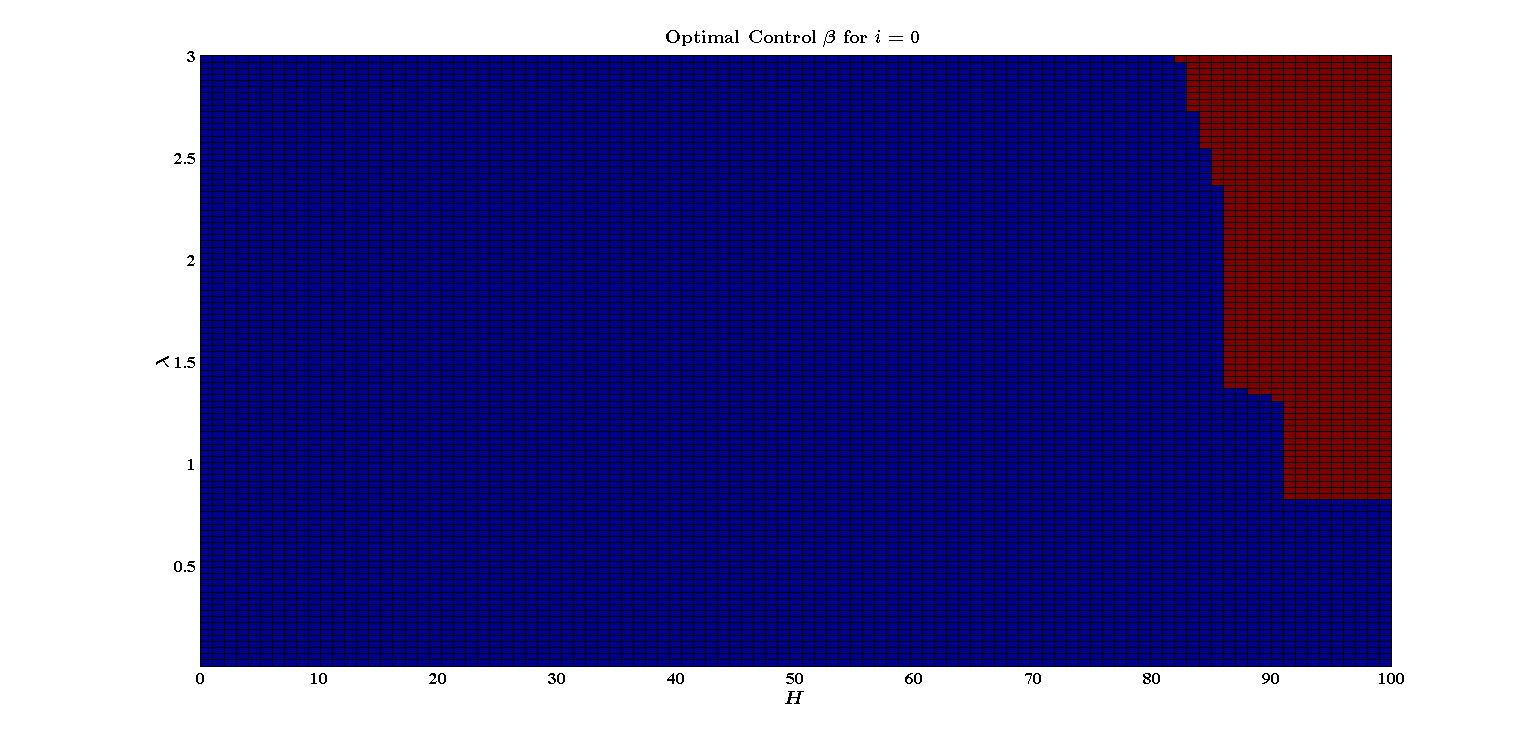} 
        \includegraphics[scale=0.38]{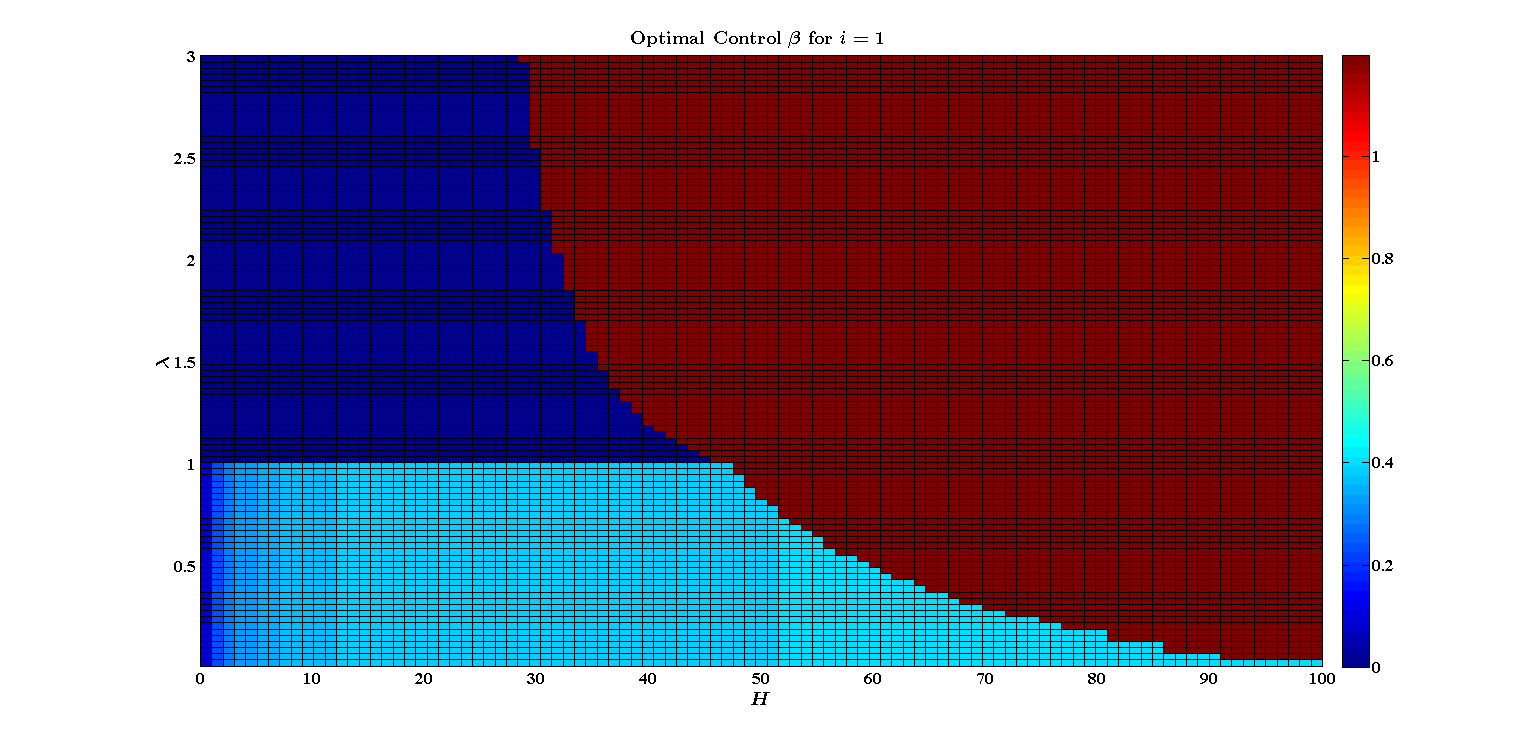} 
    \caption{\emph{Optimal control $\beta$ for $i=0$ and $i=1$.}}
    \label{Fig3_R}
      \end{figure}
\end{center}

Finally, we study the sensitivity of the optimal policy. Since our goal is to evaluate the impact of climate change, we focus on the effect of the parameter $c$, which describes the self-excitation of the process. This parameter captures both the occurrence of extreme rainfall events, and thus the risk of dam overtopping, and longer dry spells, which increase the risk of water shortages.
As a consequence, there is a dilemma for the manager: reducing the water level decreases the risk of overtopping but increases the cost of low water and the likelihood of incurring high costs due to water scarcity. Conversely, increasing the water level reduces the risk of water shortages but raises the likelihood of dam overtopping, thus taking on a significant risk.
In the limit when $c=0$, the Hawkes process becomes a standard Poisson process. 
To evaluate the impact of the parameter $c$, we examine the water level at which it is optimal to open the spillover system. This can be interpreted as the optimal water level inside the artificial lake for any fixed intensity $\lambda$. Figure \ref{Fig4_R} shows the optimal water level as $c$ moves from 
 $1$ to $1 e^{-4}$ (i.e. $c=1, 0.1, 0.01, 0.001, 0.0001$). We observe that the optimal water level is decreasing in $c$ (i.e. increasing with the plots) showing that the dam overtopping risk dominates over the water shortages. This is more pronounced for intermediate values of $\lambda$. 

\begin{center}
\begin{figure}[h!]
    \centering
     \includegraphics[scale=0.24]{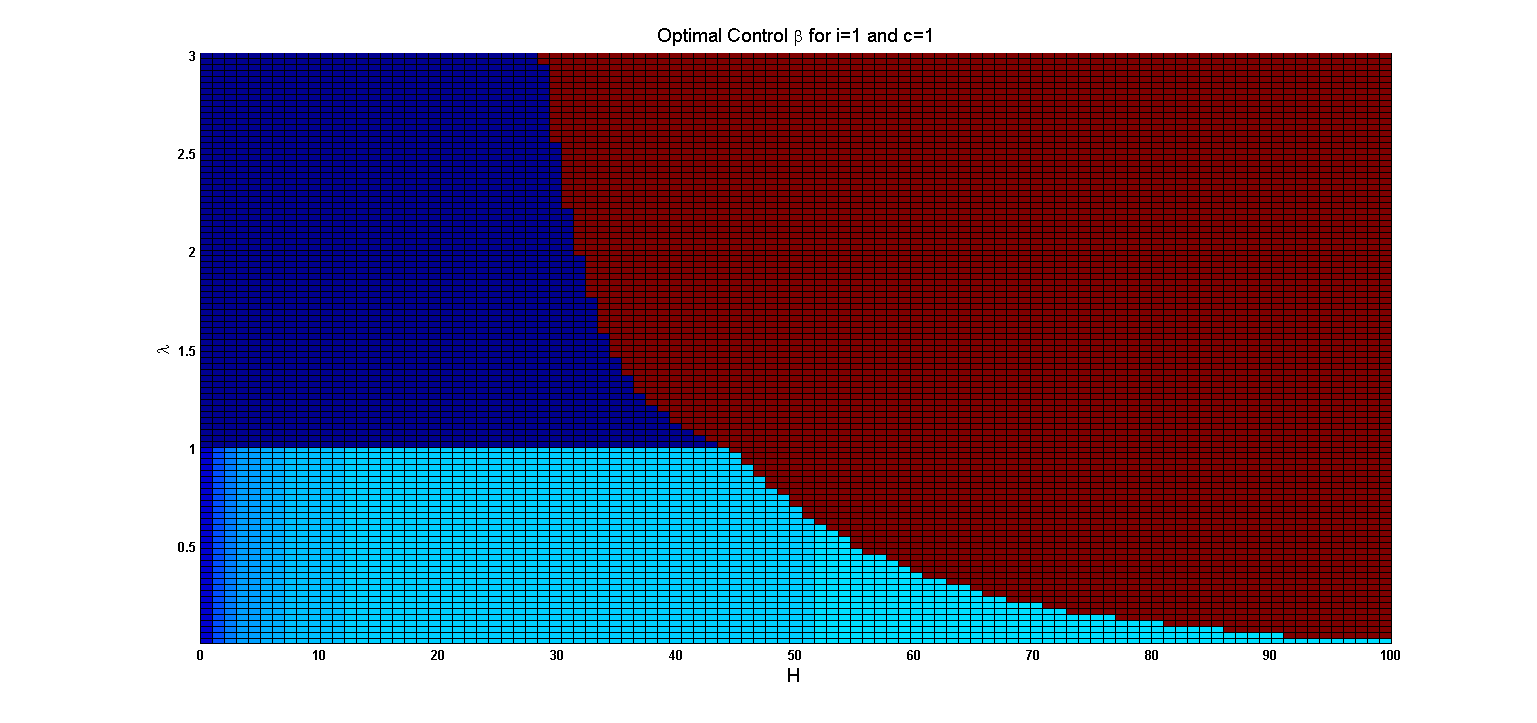} 
 \includegraphics[scale=0.24]{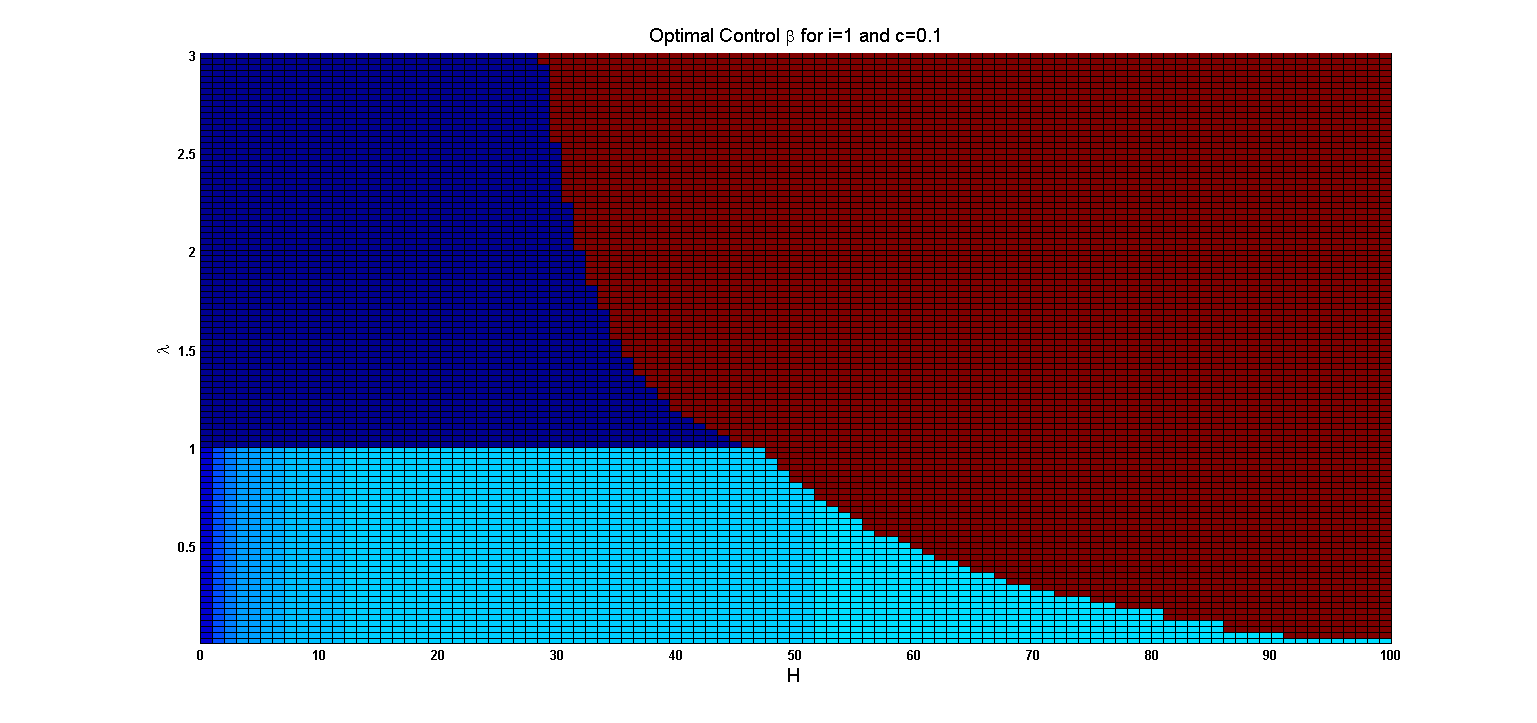}
  \includegraphics[scale=0.24]{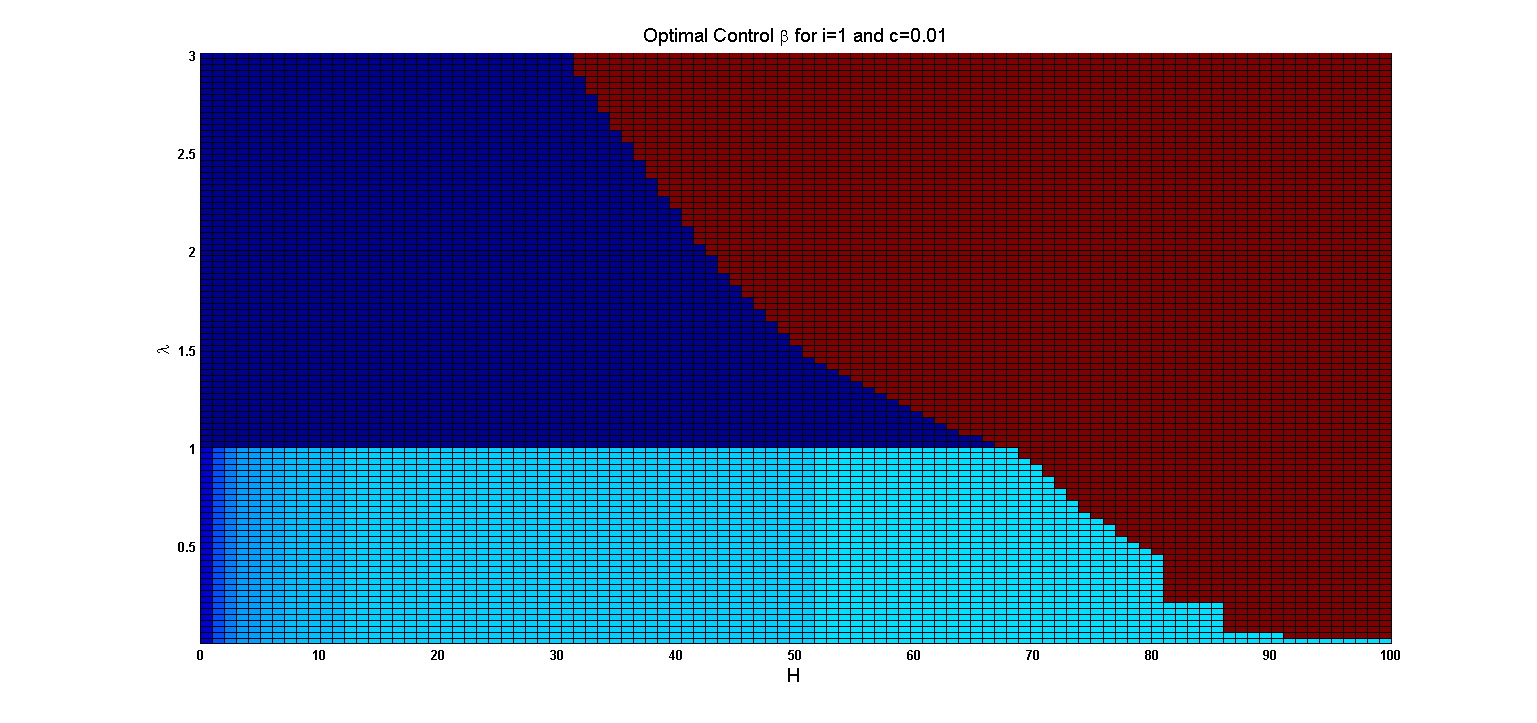}
   \includegraphics[scale=0.24]{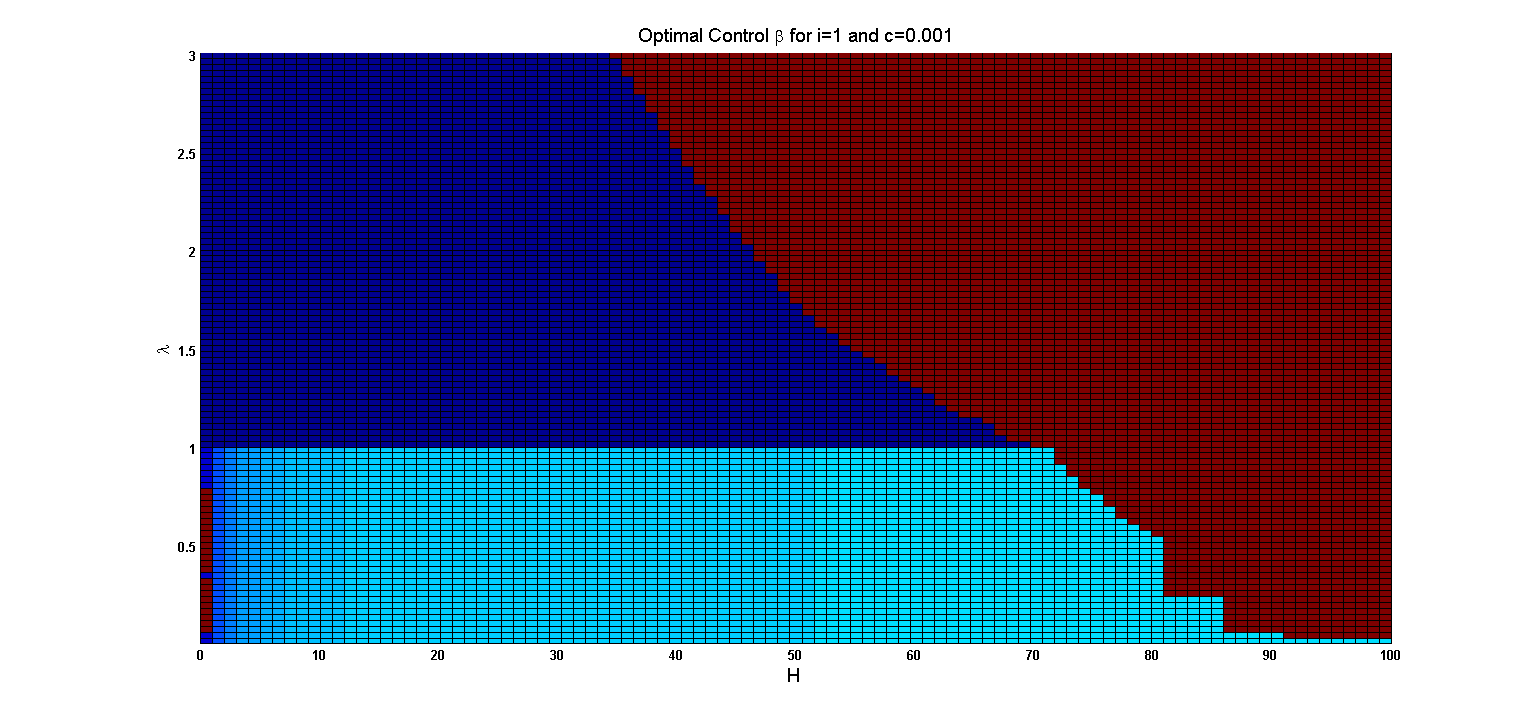}
    \includegraphics[scale=0.24]{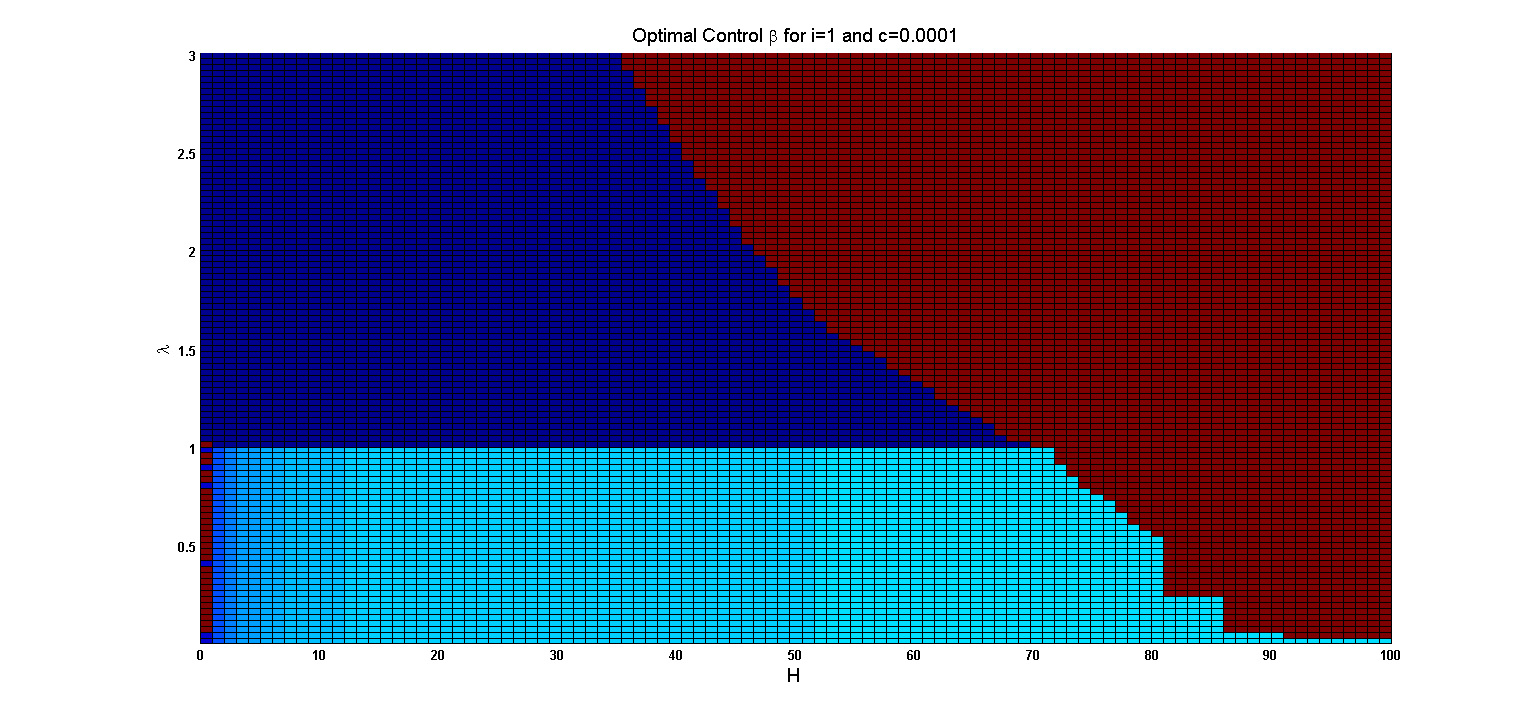}    
    \caption{\emph{Optimal control $\beta$ in case $i=1$ for $c= 1, 0.1, 0.01, 0.001, 0.0001$.}}
    \label{Fig4_R}
      \end{figure}
\end{center}

Finally we present some trajectories under the optimal controls, their plots are in Figure \ref{Fig5}

\begin{Remark}
We observe that the irregular patterns observed near the boundary are numerical artefacts associated
with the Neumann boundary conditions and not economically meaningful features of
the optimal policy. In fact boundary numerical artifacts often arise due to the fact that Neumann-type boundary conditions are imposed. Reflective conditions frequently lead to ill-posedness issues. Moreover, the absence of Brownian terms, and thus of second-order derivatives, does not help to regularize the problem. 

\end{Remark}
\begin{center}
\begin{figure}[h]
    \centering
        \includegraphics[scale=0.45]{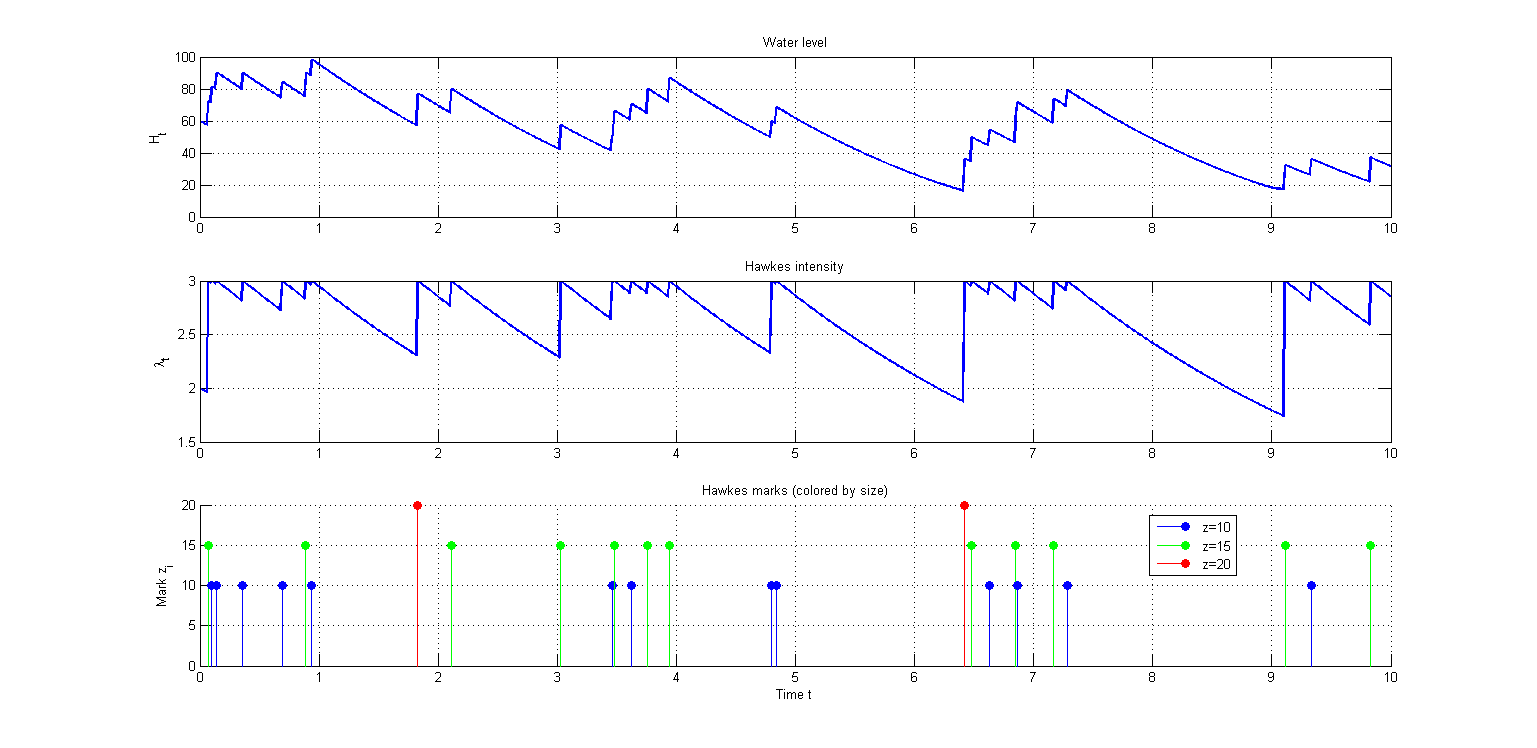} 
             \includegraphics[scale=0.45]{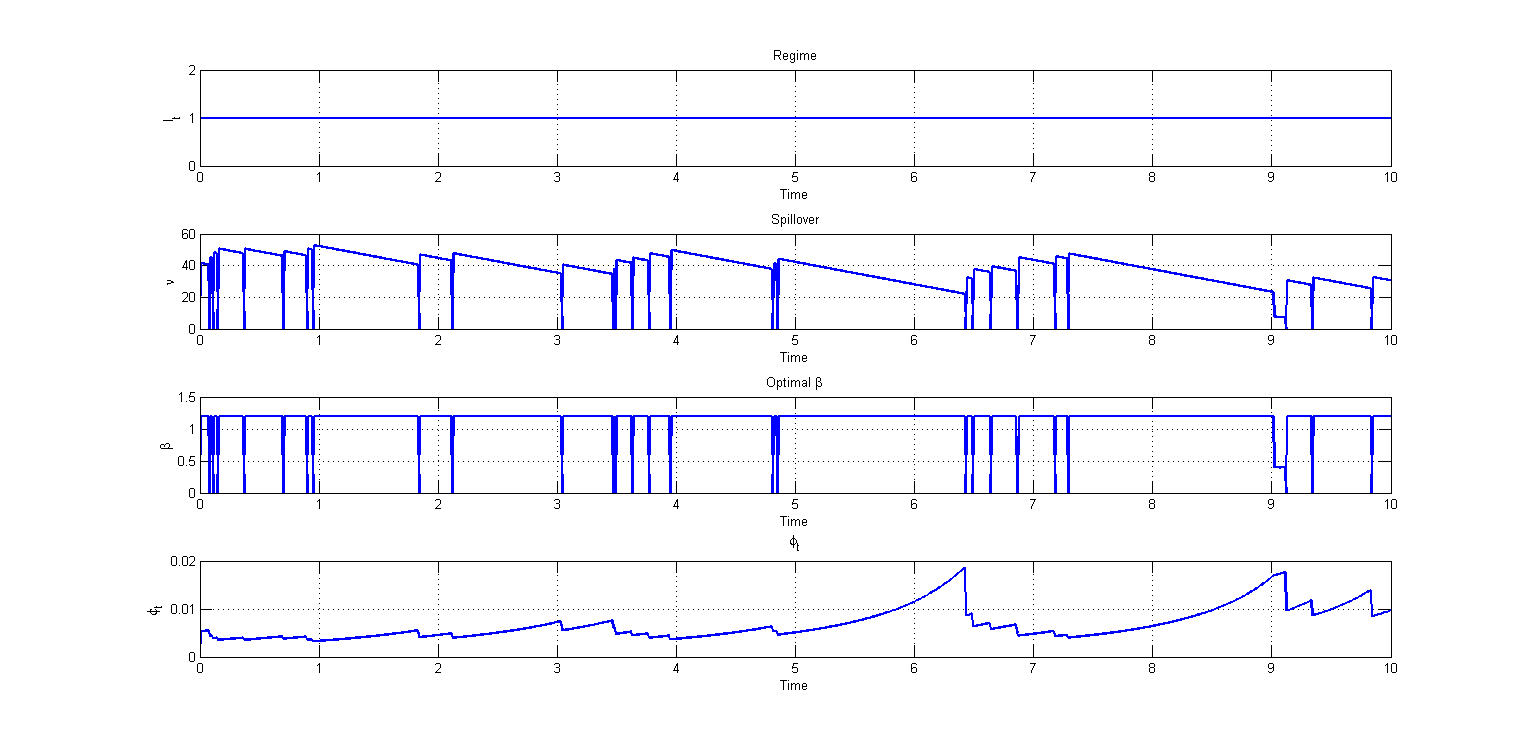} 
    \caption{\emph{Optimal trajectories and optimal controls}}
    \label{Fig5}
      \end{figure}
\end{center}

\section{Concluding Remarks}\label{concluding-remarks}

Climate deregulation has a profound impact on the frequency and quantity of rainfall, even in temperate zones. The main consequence is that the usual diffusion/Gaussian models for water inside reservoirs, as discussed in \cite{Brockwell, Bather, Cuvelier, Hottovy}, are no longer valid. This is because many temperate areas, such as Western Europe, California, and Eastern Australia, experience both intense rainfall episodes, often referred to as intense rainfall events, and very long periods of drought.
These new empirical facts cannot be captured by a standard Brownian diffusion, which is characterised by an infinitely divisible law with respect to time.

In this paper, we have proposed a new framework based on marked Hawkes processes. Hawkes processes \cite{Haw71} are point processes with the useful property of being self-exciting. They can capture very heavy rainfall events that have a macroscopic impact, and, at the same time, the inter-event times are dichotomous, either very short or extremely long.
They was deeply exploited in the biological \cite{Haw71}, seismic activity \cite{Ogata}
and more recently financial literature \cite{ACL15, BCCS, CMS, EGG10}.

Dam management is crucial for electricity production, with hydropower gaining importance due to the need to reduce reliance on coal and oil for ecological reasons and to offset the instability of solar and wind power generation. Electricity prices significantly influence the economic viability of hydropower. This paper focuses on the optimal policy for opening and closing turbines to maximize electricity production, while considering physical and regulatory constraints. 
Additionally, we introduce a preliminary climate impact model to address the growing influence of climate change on dam operations. Rising temperatures and shifting precipitation patterns are expected to alter reservoir inflows, with earlier snow-melts and more frequent droughts and intense rainfall events (``water bombs"). These changes, coupled with increasing agricultural water demands, will necessitate stricter government regulations for sustainable dam management. 
To the best of our knowledge, this is the first study to explicitly incorporate the physical constraints and the climate change impacts  of dam operations into the analysis. This framework formulates the problem as a switching stochastic control problem, a model widely studied in finance for productivity and debt management, as well as for other shared resources. Our work also explores the sensitivity of dam management with respect to the self-exciting parameter, which reflects climate change through its variation. In particular, we observe that the optimal water level inside the dam decreases as the self-exciting parameter increases, especially for intermediate values of the initial intensity.
This suggests that for the parameter set considered in our numerical experiments, the effect of overtopping risk appears to dominate the effect of water shortages.
The main consequence is that dams will no longer primarily serve as a source of water supply but will instead take on a more significant role in flood protection. 

From an operational perspective, the Hawkes intensity is not directly observable and would need to be estimated from rainfall observations through filtering or likelihood-based procedures. While estimation methods for Hawkes processes are well established, their calibration in a changing climatic environment remains a challenging issue and represents an important direction for future research. 

Moreover the present sensitivity analysis focuses on the self-excitation parameter because it is the parameter most directly associated with the clustering effects induced by climate change. A broader sensitivity analysis with respect to economic and regulatory parameters is left for future investigation.

Rather than resolving the debate on optimal dam management, this study highlights critical questions about water dynamics within reservoirs due to regulatory constraints. While the detailed modeling and analysis of climate change impacts are deferred to future research, this paper underscores the urgency of integrating such considerations into dam management strategies. An important challenge for future research is the estimation and filtering of Hawkes intensities from hydrological data, as well as the calibration of self-excitation parameters in a changing climate.


\section{Acknowledgments}
Cristina Di Girolami and Simone Scotti acknowledge the support of INDAM, Instituto Nazionale di Alta Matematica Francesco Severi, and GNAMPA, Gruppo Nazionale per l'Analisi Matematica, la Probabilit\`a e le loro Applicazioni,
Project CUP E53C23001670001 \emph{Problemi di controllo ottimo stocastico con memoria a informazione parziale}. Cristina Di Girolami acknowledges the support of Ministero Italiano dell'Università e della Ricerca, through the Programma di Ricerca Scientifica
di Rilevante Interesse Nazionale 2022BEMMLZ\_003  \emph{Stochastic control and games and the role of information}. The research of M'hamed Gaigi was supported by the Tunisian Ministry of Higher Education and Research through the research project PEJC2024-D1P28. 
The research was supported by Institute Louis Bachelier, Paris.


\newpage

\begin{thebibliography}{99}



\bibitem{Abdel-Hameed} Abdel-Hameed, M., \textit{Optimal control of a dam using $P^{M}_{\lambda , \tau}$ policies and penalty cost when the input process is a compound Poisson process with positive drift.}, Journal of Applied Probability, 37(2), 408-416 (2000).

\bibitem{ACL15}
Ait-Sahalia, Y., Cacho-Diaz, J., and Laeven, R. J. (2015). Modeling financial contagion using mutually exciting jump processes. J. Financ. Econ., 117(3), 585-606.

\bibitem{Alonzo}
Alonzo, B., Ringkjob, H. K., Jourdier, B., Drobinski, P., Plougonven, R., and Tankov, P. (2017). Modelling the variability of the wind energy resource on monthly and seasonal timescales. Renewable energy, 113, 1434-1446.

\bibitem{BardiCapuzzo1997}
Bardi, M. and Capuzzo-Dolcetta, I. \textit{Optimal Control and Viscosity Solutions of Hamilton--Jacobi--Bellman Equations}. Birkh\"auser, Boston, 1997.

\bibitem{Brockwell} Attia, F.A., Brockwell, P.J. \textit{The Control of a Finite Dam.}, Journal of Applied Probability, 19(4), 815-825 (1982).

\bibitem{Bather} Bather, J.A., \textit{A Diffusion Model for the Control of a Dam.}, Journal of Applied Probability, 5(1), 55-71 (1968).

\bibitem{BerBriScoSga21} Bernis, G. Brignone, R., Scotti, S. and Sgarra, C. (2021): A Gamma Ornstein-Uhlenbeck model driven by a Hawkes process.
Math. Financ. Econ., 15, 747-773.

\bibitem{BondiPriola25} Bondi, A. and Priola, E. (2025) Preprint, arXiv:2307.16871, 2024. {\it https://arxiv.org/pdf/2307.16871}. 


\bibitem{BBH21}
Bessy-Roland, Y., Boumezoued, A., and Hillairet, C. (2021). Multivariate Hawkes process for cyber insurance. Ann. Actuar. Sci., 15(1), 14-39.

\bibitem{BCCS}
Brachetta, M., Callegaro, G., Ceci, C., and Sgarra, C. (2024). Optimal reinsurance via BSDEs in a partially observable model with jump clusters. Finance and Stochastics, 28(2), 453-495.

\bibitem{BouchTouzi2011} Bouchard, B. and Touzi, N. \textit{Weak dynamic programming principle for viscosity solutions}. SIAM Journal on Control and Optimization, 49(3), 948–962, 2011.


\bibitem{CMS}
Callegaro, G., Mazzoran, A., and Sgarra, C. (2022). A self-exciting modeling framework for forward prices in power markets. {\it Applied Stochastic Models in Business and Industry}, 38(1), 27-48.

\bibitem{HwRains}  Chapon A., Ouarda T.B.M.J. and Bertand N., \textit{Stochastic generator for rainfall with a Hawkes process marked by an extended generalized Pareto and a vine copula} Environmental Modelling and Software 191 (2025) 106490.

\bibitem{chegaily}  Chevalier, E., Gaigi, M., and Ly Vath, V. (2017).
Liquidity risk and optimal dividend/investment strategies. {\it Mathematics and Financial Economics, 11(1)}, 111-135.


\bibitem{ChLyRocSco}  Chevalier, E.,  Ly Vath, V., Roch, A., and Scotti, S. (2015). Optimal exit strategies for investment projects. {\it Journal of Mathematical Analysis and Applications, 425(2)}, 666-694.

\bibitem{ChDiGaGiSc}  Chevalier, E., Di Girolami, C., Gaigi, M, Giovannini, E. and  Scotti, S. (2023) \textit{A dam management problem with energy production as an optimal switching problem}. Applied Stochastic Modelling for Business and Industry; 1-16. doi: 10.1002/asmb.2840
%
%
%


\bibitem{Cuvelier} Cuvelier, T., Archambeau, P., Dewals, B., and Louveaux, Q.\textit{Comparison between Robust and Stochastic Optimization for Log-term Reservoir Management Under Uncertainty}, Water Resources Management, 32(5), 1599-1614 (2018).

\bibitem{DZ}
Dassios, A., and Zhao, H. (2013). Exact Simulation of Hawkes Process with Exponentially Decaying Intensity. Electronic Communications in Probability, 18(62).

\bibitem{EGG10} Errais, E., Giesecke, K., and Goldberg, L. R. (2010). Affine point processes and portfolio credit risk. SIAM J. Financial Math., 1(1), 642-665.




\bibitem{Faddy} Faddy, M.J., \textit{Optimal Control of Finite Dams: Continuous Output Procedure.} Advances in Applied Probability, 6(4), 689-710 (1974).

\bibitem{Harv} Ga\"igi, M., Ly Vath, V. and Scotti, S. \textit{Optimal harvesting under uncertain environment with clusters of catastrophes}. Journal of Economic Dynamics and Control, Volume 179, October 2025, 105165.

\bibitem{Geman}
Geman, H., El Karoui, N., and Rochet, J. C. (1995). Changes of numeraire, changes of probability measure and option pricing. Journal of Applied probability, 32(2), 443-458.

\bibitem{Haw71} Hawkes, A. G. (1971): Spectra of some self-exciting and mutually exciting point processes. Biometrika, 58(1), 83-90.

\bibitem{HLOS22}
Hillairet, C., Lopez, O., d'Oultremont, L., and Spoorenberg, B. (2022). Cyber-contagion model with network structure applied to insurance. Insur. Math. Econ., 107, 88-101.


\bibitem{HL21}
Hillairet, C., and Lopez, O. (2021). Propagation of cyber incidents in an insurance portfolio: counting processes combined with compartmental epidemiological models. Scand. Actuar. J., 2021(8), 671-694.

%
%
\bibitem{Hottovy} Hottovy, S., and Stechmann, S.N., \textit{Threshold Models for Rainfall Convection: Deterministic versus Stochastic Triggers.}, SIAM Journal of Applied Mathematics, 75(2), 861-884 (2015).

\bibitem{FlemingSoner2006}
Fleming, W.H.  and Soner., M.H. \textit{Controlled Markov Processes and Viscosity Solutions}. Springer, New York, 2nd edition, 2006.


\bibitem{Ishi2004} Ishikawa, Y. (2004)  \textit{Optimal control problem associated with jump processes.}  Applied Mathematics
and Optimization, 50, 21–65.

\bibitem{JacShi2003} J. Jacod, A.N. Shiryaev, \textit{Limit Theorems for Stochastic Processes}, Second, in: Grundlehren der
Mathematischen Wissenschaften, vol. 288, Springer-Verlag, Berlin, 2003
\bibitem{JMS17}
Jiao, Y., Ma, C., and Scotti, S. (2017). Alpha-CIR model with branching processes in sovereign interest rate modeling. Finance and Stochastics, 21(3), 789-813.


\bibitem{JMSS19}
Jiao, Y., Ma, C., Scotti, S., and Sgarra, C. (2019). A branching process approach to power markets. Energy Economics, 79, 144-156.

\bibitem{JMS21}
Jiao, Y., Ma, C., Scotti, S., and Zhou, C. (2021). The Alpha-Heston stochastic volatility model. Math. Finance, 31(3), 943-978.


\bibitem{Krach-Macrina-24}
Krach, F., Macrina, A., Kanter, A., Hampwaye, E., Hlalukana, S., and Rateele, N. T. (2024). The Financial Impact of Carbon Emissions on Power Utilities Under Climate Scenarios. International Journal of Theoretical and Applied Finance, 27(01).


\bibitem{KushDup} Kushner, H., and Dupuis, P. G. (2013). {\it Numerical Methods for Stochastic
Control Problems in Continuous Time} (Vol. 24). Springer Science \& Business Media.

\bibitem{Lando}
Lando, D. (1998). On Cox processes and credit risky securities. Review of Derivatives research, 2, 99-120.

\bibitem{LyVath-Pham}  V. Ly Vath and H. Pham, {\it Explicit solution to an optimal switching problem in the two-regime case}, SIAM J. Cont. Optim., 46, 395-426, 2007.

\bibitem{LyVath-Pham-Villeneuve} V. Ly Vath, H. Pham and S. Villeneuve, {\it A mixed singular/switching control problem for a dividend policy with reversible technology investment}, Annals of Applied Probability, 18, pp. 1164-1200, 2008.

\bibitem{GoldysWu2016} Goldys, B. and Wu, W. \textit{Dynamic programming principle for stochastic control problems driven by general L\'evy noise}. Stochastic Analysis and Applications, 34(6):1083--1093, 2016.

\bibitem{Mamon}
Mamon, R. S., and Elliott, R. J. (Eds.). (2007). Hidden Markov models in finance (Vol. 4). New York: Springer.

\bibitem{Manso}
Manso, P., Schleiss, A., Avellan, F., and  Stahly, M. (2016). Electricity supply and hydropower development in Switzerland. Proceedings HYDRO 2016.

\bibitem{McInnes} McInness, D., \textit{Optimization of Controlled Markov Chains with Application to Dam Management.}, PhD. Thesis, Monash University (2015).

\bibitem{Miller2017} McInness, D., and  Miller, B.\textit{Optimal control of a large dam using time-inhomogeneous Markov chains with an application to flood control.}, IFAC PapersOnLine, 50(1), 3499-3504 (2017).

\bibitem{Meyerhoff}
Meyerhoff, J., Klefoth, T., and  Arlinghaus, R. (2019). The value artificial lake ecosystems provide to recreational anglers: Implications for management of biodiversity and outdoor recreation. Journal of Environmental Management, 252, 109580.

\bibitem{Minghi} 
Minghi, J. (Ed.). (2011). The Structure of Political Geography (1st ed.). Routledge. https://doi.org/10.4324/9781315135267

\bibitem{MP22}
Morariu-Patrichi, M., and Pakkanen, M. S. (2022). State-dependent Hawkes processes and their application to limit order book modelling. Quant. Finance, 22(3), 563-583.


\bibitem{MoralesETC} 
Morales, F. E. C., Batista do Nascimento, A. M., Paez, M. S., Rodrigues, D. T.,  Apolinario, C. D. M. (2025). Improvement of the Rnnmm type climate index approach with a spatio-temporal model based on the Hawkes process. EGUsphere, 2025, 1-27.



\bibitem{Ogata}
Ogata, Y. (1988). Statistical models for earthquake occurrences and residual analysis for point processes. Journal of the American Statistical association, 83(401), 9-27.

\bibitem{OksendalSulem2019} \O ksendal, B. and Sulem, A. \textit{Applied Stochastic Control of Jump Diffusions}. Springer, Cham, 3rd edition, 2019.

\bibitem{Pham2009} Pham, H. \textit{Continuous-time Stochastic Control and Optimization with Financial Applications}. Springer, Berlin, 2009. 

\bibitem{Prasanchum} Prasanchum, H., and Kangrang, A.\textit{Optimal reservoir rule curves under climatic and land use changes for Lampao Dam using Genetic Algorithm.}, KSCE Journal of Civil Engineering, 22(1), 351-364 (2018).


\bibitem{Protter} Protter P.E., \textit{Stochastic Integration and Differential Equations}, 2nd Edition, vol. 21 of Applications of Mathematics (New York). Springer-
Verlag, Berlin, 2004. Stochastic Modelling and Applied Probability.

\bibitem{Roche}
Roche, C., Thygesen, K., and Baker, E. (Eds.) 2017. Mine
Tailings Storage: Safety Is No Accident. A UNEP Rapid
Response Assessment. United Nations Environment
Programme and GRID-Arendal, Nairobi and Arendal,
ISBN: 978-82-7701-170-7


%
%

\bibitem{Shardin} Shardin, A.A., and Wunderlich, R., \textit{Partially observable stochastic optimal control problems for an energy storage.}, Stochastics, 89(1), 280-310(2017).

\bibitem{Soner1988} Soner, H. M. \textit{Optimal control of jump-{M}arkov processes and viscosity solutions}.
\newblock In Fleming, W. H. and Lions, P.-L. editors, {\em Stochastic Differential Systems, Stochastic Control Theory and Applications}, volume~10 of {\em The IMA Volumes in Mathematics and its Applications}, pages 501--511. Springer, New York, 1988.


\bibitem{Timmermann}
Timmermann, A., An, S. I., Kug, J. S., Jin, F. F., Cai, W., Capotondi, A., ... Zhang, X. (2018). El Nino southern oscillation complexity. Nature, 559(7715), 535-545.

\bibitem{Tingsanchali}
Tingsanchali, T., and Chinnarasri, C. (2001). Numerical modelling of dam failure due to flow overtopping. Hydrological sciences journal, 46(1), 113-130.

\bibitem{Touzi}
Touzi, N. (2012). Optimal stochastic control, stochastic target problems, and backward SDE (Vol. 29). Springer Science \& Business Media.

\bibitem{Yeh1} Yeh, L., \textit{Optimal Control of a Finite Dam: Average-Cost Case.} Journal of Applied Probability, 22(2), 480-484 (1985).

\bibitem{Yeh2} Yeh, L., and Hua, L.J., \textit{Optimal Control of a Finite Dam: Wiener Process Input.}, Journal of Applied Probability, 24(1), 186-199 (1987).

%

\bibitem{YY22a} Yoshioka, H., and Yoshioka, Y. (2022). Modeling Clusters in Streamflow Time Series Based on an Affine Process. In Modeling, Simulation and Optimization (pp. 379-385). Springer, Singapore.


\bibitem{YongZhou1999} Yong, J. and Zhou, X. Y. \textit{Stochastic Controls: Hamiltonian Systems and HJB Equations}.
Springer, New York, 1999. 

\bibitem{YY22b} Yoshioka, H., and Yoshioka, Y. (2022). Stochastic streamflow and dissolved silica dynamics with application to the worst-case long-run evaluation of water environment. Optim. Engin., 1-34.



\bibitem{Zuckerman} Zuckerman, D., \textit{Characterization of the optimal class of output policies in a control model of a finite dam.}, Journal of Applied Probability, 18(4), 913-923 (1981).


\end{thebibliography}
\end{document}